\documentclass[11pt,twoside]{amsart} 

  \usepackage{url}
  \usepackage{amssymb}

  \usepackage[a4paper, margin=3.5cm]{geometry}

  \input xypic
  \xyoption{all}

 \usepackage{times}
 
  \usepackage{amsmath, amsthm, amsfonts}

  \theoremstyle{plain}
  \newtheorem{theorem}{Theorem}[section]
  
  \newtheorem{proposition}[theorem]{Proposition}
  
  \newtheorem{corollary}[theorem]{Corollary}

  \theoremstyle{definition}
  \newtheorem{definition}[theorem]{Definition}

  \newtheorem{remark}[theorem]{Remark}

  \numberwithin{equation}{section}
  \numberwithin{figure}{section}

  \renewcommand{\cH}{{\mathcal H}}
  
  \newcommand{\cK}{{\mathcal K}}
  
  \newcommand{\cA}{{\mathcal A}}

  \newcommand{\cP}{{\mathcal P}}

   \newcommand{\ba}{\begin{eqnarray}}
   \newcommand{\na}{\end{eqnarray}}
   \newcommand{\ban}{\begin{eqnarray*}}
   \newcommand{\nan}{\end{eqnarray*}}

     \newcommand{\vlim}{\varinjlim}

  \newcommand{\CC}{{\mathbb C}}
  \newcommand{\RR}{{\mathbb R}}
  \newcommand{\ZZ}{{\mathbb Z}}
  
  \newcommand{\KK}{{\mathbb  K}}

 \newcommand{\CL}{{\bf  Cliff}}
  \newcommand{\ind}{{\bf  Index }}
  \newcommand{\MSpin}{{\bf MSpin^c }}
  \newcommand{\Th}{{\bf  Th}}
\newcommand{\BSpin}{{\bf  BSpin}}
 \newcommand{\BString}{{\bf  BString}}
  \newcommand{\BSO}{{\bf  BSO}}
   \newcommand{\MString}{{\bf  MString}}
   \newcommand{\MSp}{{\bf  MSpin}}
   \newcommand{\MSO}{{\bf  MSO}}
    \newcommand{\String}{{\bf  String}}
   \newcommand{\Fred}{{\bf  Fred}}
 
  \renewcommand{\a}{\alpha}
  \renewcommand{\b}{\beta}
   \renewcommand{\i}{\iota}
  
  \newcommand{\eps}{\epsilon}
  
  \newcommand{\pa}{\partial}

    \newcommand{\disp}{\displaystyle}

  \newcommand{\nin}{\noindent}

  \def\cancel#1#2{\ooalign{$\hfil#1\mkern1mu/\hfil$\crcr$#1#2$}}
\def\Dirac{\mathpalette\cancel D}

\begin{document}

\title[Geometric cycles, index theory and  twisted K-Homology]
{Geometric cycles, index theory and  twisted K-homology  }

  \author[Bai-Ling Wang]{Bai-Ling Wang}
  \address[Bai-Ling Wang]
  {Department of Mathematics\\
  Mathematical Sciences Institute\\
  Australian National University\\
  Canberra ACT 0200 \\
  Australia}
  \email{wangb@maths.anu.edu.au}

\date{September 25, 2007}

  \thanks{
 }

  \begin{abstract} We study twisted $Spin^c$-manifolds over a paracompact Hausdorff
  space $X$ with a twisting  $\alpha: X \to K(\ZZ, 3)$. We introduce   the topological  index and
  the analytical index on the bordism group of $\alpha$-twisted   $Spin^c$-manifolds 
  over $(X, \alpha)$,  taking values in topological twisted K-homology
  and analytical  twisted K-homology respectively.  The main result of this paper is to  establish the equality between  the topological index and the analytical index for smooth manifolds.   We also  define a notion of  geometric twisted K-homology, whose cycles are  geometric cycles  of $(X, \a)$ analogous to 
  Baum-Douglas's geometric cycles.
  As an application of our twisted index theorem,  we discuss the twisted longitudinal index theorem for a foliated manifold $(X, F)$ with a twisting  $\alpha: X \to K(\ZZ, 3)$,  which generalizes  the Connes-Skandalis index theorem for foliations and  the  Atiyah-Singer families index theorem to twisted cases. 
  \end{abstract}

\maketitle

\tableofcontents

\newpage

\section{Introduction}

According to work of Baum and Douglas \cite{BD1} \cite{BD2}, the  Atiyah-Singer index theorem (\cite{AtiSin1}\cite{AtiSin3})  for a closed smooth manifold $X$  
can be formulated  as in the  following commutative  triangle 
\ba\label{index:BD1}
   \xymatrix{ 
& K^0 (T^*X)  \ar[dl]_{\ind_t} \ar[dr]^{\ind_a} 
&\\
K^{t}_0(X) \ar[rr]^{\mu}_{\cong} & & K^a_0 (X ),}
\na
 whose arrows are all   isomorphisms.  Here $K^0 (T^*X) $ denotes the K-cohomology with compact supports of the cotangent
 bundle $T^*X$, 
 corresponding to symbol classes of  elliptic pseudo-differential  operators on $X$. $K^{t}_0(X)$ is the topological K-homology  constructed in \cite{BD1}, and $K^a_0 (X )$
 is the Kasparov's analytical K-homology (see \cite{Kas} and \cite{HigRoe}) of the $C^*$-algebra $C(X)$ of continuous complex-valued functions 
 on $X$. 
 
 The topological index    and the analytical index   can defined on the level
 of cycles.   The basic cycles for $ K^t_0 (X) $ (resp. $K^t_1(X)$)  are triples $(M, \i, E)$
  consisting of even-dimensional (resp. odd-dimensional)  closed smooth manifolds $M$  with a
 given $Spin^c$ structure on the tangent bundle of $M$ together with a continuous map $\i: M\to X$
 and a complex vector bundle $E$ over $M$. The equivalence relation on the  set  of all cycles is generated by
 the following three steps (see \cite{BD1} for details):
 \begin{enumerate}
\item[(i)] Bordism
\item[(ii)] Direct sum and disjoint union
\item[(iii)] Vector bundle modification.
 \end{enumerate}
Addition in $K^t_{ev/odd}(X)$ is given by the disjoint union operation of topological cycles. 

 A symbol class  in $K^0(T^*X)$  of an  elliptic  pseudo-differential   operator $D$  on $X$ is  represented by 
\[
\sigma(D): \pi^*E_0 \longrightarrow  \pi^*E_1
\]
where  $\pi: T^*X \to X$ is the projection, $E_0$ and $E_1$ are complex vector bundles over $X$. Choose a Riemannian metric on $X$,  let $S(T^*X \oplus \underline{\RR})$
be  the unit sphere bundle in $T^*X \oplus \underline{\RR}$, equipped with the natural 
$Spin^c$ structure. Denote by $\phi$   the projection $ S(T^*X \oplus \underline{\RR}) \to X$.
Let $\hat{E}$ be  the complex vector bundle over $S(T^*X \oplus \underline{\RR})$  obtained
by the clutch construction (See section 10 in \cite{BD1}): as 
$S(T^*X \oplus \underline{\RR})$  consists of two copies of the unit ball bundle
of $T^*X$ glued together along the unit sphere bundle, one can use the
symbol $\sigma (D) $ to clutch $\pi^*E_0$ and $\pi^*E_1$ together along 
the unit sphere bundle  $S(T^*X)$. The topological  index $\ind_t([\sigma(D)])$ is represented by the
following    topological cycle
\[
(S(T^*X \oplus \underline{\RR}), \hat{E}, \phi). 
\]

The Kasparov's  analytical K-homology   $K^a_{ev/odd}(X)$,  denoted $KK^{ev/odd}(C(X), \CC)$ in the literature, is 
generated by unitary equivalence classes of multi-graded Fredholm modules  over $C(X)$ modulo operator homotopy
relation \cite{Kas}. Addition in $KK^{ev/odd}(C(X), \CC)$  is defined using a natural notion of direct sum  
of Fredholm modules, see \cite{HigRoe} for details.  The analytical index $\ind_a([\sigma(D)])$ is defined in terms of   
Poincar\'{e} duality  (Cf. \cite{Kas1})
  \[\begin{array}{lllll}
K^0(T^*X)&  \cong & KK^0(\CC, C_c(T^*X))  && \\[2mm]
& \cong & KK^0(C(X), \CC) & \quad & \text{(Kasparov's Poincar\'{e} duality)} \\[2mm]
&=& K^a_0(X).&&  \end{array}
\]

 On the level of cycles, an even dimensional topological cycle $(M, f, E)$
 defines a canonical element $[\Dirac^E_M]$  in $K^a_0(M)$ determined  by the Dirac operator 
 \[
 \Dirac^E_M:  C^\infty(S^+\otimes E) \longrightarrow C^\infty(S^-\otimes E)
 \]
 where $S^\pm$are  the positive and negative spinor bundles (called reduced spinor bundles
 in \cite{HigRoe}). Then the natural isomorphism 
 \[
 \mu:  K^t_0(X) \longrightarrow K^a_0(X)
 \]
 is defined by the  correspondence 
 \[
 (M, \i, E) \mapsto \i_*([\Dirac^E_M])
 \]
 where $\i_*: K^a_0(M) \longrightarrow K^a_0(X)$ is the covariant homomorphism induced by $\i$.
 
 The commutative triangle (\ref{index:BD1}) has played an  important role in the understanding of the Atiyah-Singer
 index theorem and its various generalizations  such as the Baum-Connes conjecture  in \cite{BC1}. 
 In this paper, we will generalize  the Atiyah-Singer index theorem to the framework of twisted K-theory following 
ideas inspired from Baum-Douglas \cite{BD1} \cite{BD2} and Baum-Connes \cite{BC1}.  

In this paper, we aim to develop  the   index theorem  in the framework of twisted K-theory which
is a natural generalization of the Baum-Douglas  commutative triangle  (\ref{index:BD1}). 
 We need a notion of 
 a twisting in complex K-theory, given by a continuous map
 \[
 \a: X\longrightarrow K(\ZZ, 3)
 \]
 where $K(\ZZ, 3)$ is an Eilenberg--MacLane space. We often  choose    a homotopy model of  $K(\ZZ, 3)$ as the classifying space of the 
 projective unitary group $PU(\cH)$ of an  infinite dimensional, complex and separable Hilbert space $\cH$,
 equipped  with the norm topology. 
The norm topology could be too restrictive for some examples, one might have   to use the compact-open topology instead as discussed in \cite{AS1}.     

For any paracompact Hausdorff space $X$ with a  continuous map  $\a: X\to K(\ZZ, 3)$,    the corresponding
 principal $K(\ZZ, 2)$-bundle  over $X$ will be denoted by $\cP_\a$. Then any
base-point preserving 
action of $K(\ZZ, 2)$ on a  spectrum  defines an associated 
bundle of based spectra.  In this paper, we mainly consider two spectra, one is the complex
K-theory spectrum $\KK = \{\Omega^n\KK\}$, the other is the $Spin^c$ Thom spectrum $\MSpin= \{  \MSpin (n)\}$.  The corresponding
bundle of based spectra over $X$ will be denoted by
\[
 \cP_\a(   \KK)     \quad  \text{and} \quad \cP_\a(\MSpin)  
\]
respectively.

Twisted K-cohomology groups of  $(X, \a)$  are  defined to be 
\[
\pi_0\bigl( C_c (X, \cP_\a(\Omega^n  \KK)) \bigr)
\]
the homotopy classes of compactly supported sections of  $\cP_\a(\Omega^n  \KK)$.  

 Let $\cK$ be    the $C^*$-algebra  of  compact 
 operators on the Hilbert space $\cH$, and  $\cP_\a(\cK)$ be the associated bundle  of compact operators 
  corresponding to the 
 $PU(\cH)$-action on  $\cK$ by conjugation. 
 An equivalent definition of  twisted K-theory  of $(X, \a) $ is  
  the algebraic K-cohomology 
groups of the continuous trace $C^*$-algebra over $X$ of  compactly supported  sections of  
$\cP_\a( \cK)$.
  The Bott periodicity of the K-theory
spectrum implies that we only have two  twisted K-groups,   denoted by
\[
K^0(X, \a) \ \text{and} \  K^1(X, \a),
\]
or simply  $K^{ev/odd}(X, \a)$. 
We will review  twisted K-theory and its basic properties in section \ref{review:K}.

We  define   topological twisted K-homology  to be
\[
K^t_{n} (X, \a) :=\disp{ \vlim_{k\to \infty}} [S^{n+2k},  \cP_\a(\Omega^{2k} \KK)/X]
\]
the stable homotopy groups of $\cP_\a(\KK)/X$. 

Due to  Bott periodicity, we only have two different topological twisted K-homology groups denoted by
$K^t_{ev/odd}(X, \a)$.  There is a notion of analytical twisted K-homology  defined as  Kasparov's  analytical K-homology
\[
K^a_{ev/odd}(X, \a)  := KK^{ev/odd} (C_c(X, \cP_\a( \cK)), \CC).
\]
Now we can state the main theorem of this paper, which should be thought of as the general index theorem in the framework of twisted K-theory.

\vspace{3mm}
\nin{\bf Main Theorem}  (Cf. Theorem \ref{Main} and Remark \ref{remark:index})  {\em Let $X$ be a smooth manifold and
$\pi: T^*X \to X$ be the projection ,  there is a natural
isomorphism 
\[
\Phi:   K^t_{ev/odd}(X, \a) \longrightarrow K^a_{ev/odd}(X, \a), 
\]
and   there exist  notions of   the topological index and the analytical index  on 
$$K^{ev/odd}( T^*X, \a\circ\pi)$$
such that  the following diagram 
\[\xymatrix{
& K^{ev/odd} (T^*X, \a  \circ  \pi ) \ar[dl]_{\ind_t} \ar[dr]^{\ind_a}
&\\
K^t_{ev/odd} (X,  \a ) \ar[rr]^{\cong}_{\Phi} &  & K^a_{ev/odd} (X, \a )}
   \]
is commutative, and all arrows are isomorphisms.} 
\vspace{3mm}

We remark  that 
topological and analytical twisted K-homology groups are well-defined for any paracompact 
Hausdorff space $X$ with a continuous map $\a: X\to K(\ZZ, 3)$. The above main theorem  only holds for
smooth manifolds,  we believe that  the isomorphism  
\[
\Phi:   K^t_{ev/odd}(X, \a) \longrightarrow K^a_{ev/odd}(X, \a), 
\]
should be true for more general  spaces such as paracompact Hausdorff spaces  
with the homotopy type of  finite CW complexes.   We only establish
 this isomorphism  for smooth manifolds by applying the Poincar\'{e} duality  in  twisted K-theory   which requires  differential structures, see the proof of Theorem \ref{Main} for details. It would be interesting 
to have   this isomorphism for  paracompact 
Hausdorff spaces with  the homotopy type of finite CW complexes.

To prove this main theorem, we introduce a notion of  
 $\a$-twisted $Spin^c$   manifolds   over  any
 paracompact Hausdorff space $X$ with a  continuous map  $\a: X\to K(\ZZ, 3)$  in Section \ref{section:3},   which consists  of   quadruples  $(M, \nu,  \i, \eta)$ 
where 
\begin{enumerate}
\item $M$ is a smooth,  oriented and compact manifold together  with a fixed classifying map of its  stable normal bundle  
\[
\nu:   M \longrightarrow  \BSO
\]
here $\BSO= \disp{ \varinjlim_{k } } \BSO(k)$ is  the classifying space of stable normal bundle of $M$;
\item $\i: M\to X$ is a  continuous map;
\item $\eta$  is an $\a$-twisted $Spin^c$  structure  on $M$, that is   a homotopy commutative diagram
(see  Definition \ref{twisted:cob} for details)
\[
\xymatrix{M \ar[d]_{\i} \ar[r]^{\nu} &
\BSO
 \ar@2{-->}[dl]_{\eta} \ar[d]^{W_3} \\
X \ar[r]_\a  & K(\ZZ, 3). } 
\]
\end{enumerate}
A manifold $M$ admits an $\a$-twisted $Spin^c$ structure   if and only if  there exists
a continuous map
$\i: M\to X$ such that 
\[
\i^*([\a]) +  W_3(M)=0
\]
in $H^3(M, \ZZ)$. (This  is  known to physicists as the Freed-Witten anomaly cancellation
condition for Type II D-branes (Cf. \cite{FreWit}) ).

We then define an analytical index for each  $\a$-twisted $Spin^c$   manifold over $X$ taking values
in the analytical twisted K-homology $K^a_{ev/odd} (X, \a )$ and establish
its  bordism invariance. 

  In Section  \ref{section:4}, we study the  geometric  $\a$-twisted bordism groups $\Omega_*^{Spin^c} (X, \a)$ and
establish  a generalized Pontrjagin-Thom isomorphism (Cf. Theorem \ref{P-Thom}) between our geometric  $\a$-twisted bordism groups  
and the homotopy theoretic   definition of $\a$-twisted bordism groups 
\[
\Omega_n^{Spin^c} (X, \a) \cong \disp{\lim_{k\to \infty} } \pi_{n+k} (\cP_\a( \MSpin (k)) /X).
\]
We also define a topological index on geometric  $\a$-twisted bordism groups.  Then the main theorem is proved in Section
\ref{section:5}.

In Section \ref{section:6}, we explain  the notion  of  geometric cycles for any paracompact Hausdorff space $X$ with a  continuous map  $\a: X\to K(\ZZ, 3)$. Geometric cycles  in this sense  are called `D-branes' in string theory.     These consist  
of  an $\a$-twisted $Spin^c$   manifold $M$    over $X$ together with an ordinary K-class $[E]$.  Following the work
of Baum-Douglas, we impose an equivalence relation, generated
by\begin{enumerate}
\item[(i)]
Direct sum and  disjoint union
 \item [(ii)]  Bordism 
 \item[(iii)]   $Spin^c$ vector bundle 
modification\end{enumerate}
on  the set of all geometric cycles  to obtain the 
geometric twisted K-homology $K^{geo}_{ev/odd}(X, \a)$. Then we establish the commutative diagram (Cf. Theorem \ref{Main:2}) for a smooth manifold $X$ with a twisting $\a: X\to K(\ZZ, 3)$
\ba\label{equ:triangle}
\xymatrix{
& K^t_{ev/odd} (X, \a ) \ar[dl]_{\Psi} \ar[dr]^{\Phi}
&\\
K^{geo}_{ev/odd} (X,  \a ) \ar[rr]^{\cong}_{\mu} &  & K^a_{ev/odd} (X, \a )}
   \na
whose arrows are all isomorphisms. 
One consequence of this commutative diagram is    that every twisted K-class in $K^{ev/odd}(X, \a)$
can be realized by appropriate geometric cycles (Cf. Corollary \ref{D:brane}).  

  We remark that the commutative diagram (\ref{equ:triangle}) of isomorphisms should hold
for general paracompact Hausdorff spaces  with the homotopy type of  finite CW complexes. 
The restriction to smooth manifolds  is due to the fact that we  only establish
the isomorphism $\Phi$ in Theorem \ref{Main}  for smooth manifolds.  We expect that the    equivalence 
of geometric, topological and analytical twisted K-homology exists  for any finite  CW complex. 
We will return to this in a   sequel  paper  (\cite{Baum-W}.

In Section \ref{section:app}, we study the  twisted longitudinal index theorem (Cf. Theorem \ref{index:foliated})  for a foliated manifold $(X, F)$ with a twisting $\a: X\to K(\ZZ, 3)$, and show that   this twisted longitudinal index theorem 
generalizes both the Atiyah-Singer families index theorem in \cite{AtiSin4} and 
  Mathai-Melrose-Singer index theorem for projective families of elliptic operators 
  associated to a torsion twisting in \cite{MMSin}.

  In Section 8, we introduce  a  notion
of twisted $Spin$ manifolds over a manifold $X$ with a KO-twisting 
$\a: X \to K(\ZZ_2, 2)$. 
A smooth manifold $M$ admits an $\a$-twisted $Spin$ structure   if and only if  there exists
a continuous map
$\i: M\to X$ such that 
\[
\i^*([\a]) +  w_2 (M)=0
\]
in $H^2(M, \ZZ_2)$, here $w_2(X) $ is the second  Stiefel-Whitney  class of $TM$.  (This  is   the   anomaly cancellation
condition for Type I D-branes (Cf. \cite{Wit3}) ). 
We also discuss   a notion of twisted string manifolds over  a manifold $X$ with a string twisting 
$
\a: X \to K(\ZZ, 4).  
$
A smooth manifold $M$ admits an $\a$-twisted string  structure if and only if  there is a continuous map
$\i: M\to X$ such that
\[
\i^*([\a]) +    \frac{p_1 (M)}{2}  =0
\]
in $H^4(M, \ZZ)$, here $ p_1(X) $ is the first Pontrjagin class of $TM$.  
These notions  could be useful in the study  of  twisted elliptic cohomology.  

It would be interesting to establish a  local index theorem in the framework of twisted K-theory in which differential twisted K-theory  in \cite{CMW} \cite{HopSin} will come into play. We will return to these problems  in   subsequent work.  Finally, we like to point that, except in Sections  5  and  7 and Theorem
\ref{Main:2} where $X$ is smooth, $X$ is assumed to be a paracompact  Hausdorff topological space
throughout  this paper.

\section{Review of twisted K-theory} \label{review:K}

 In this section, we briefly review some basic facts about twisted K-theory, the main reference 
 are \cite{AS1} and \cite{CW2} (See also  \cite{BCMMS} \cite{FHT}  \cite{Ros}).
 
 Let $\cH$ be an  infinite dimensional, complex and separable Hilbert space. We shall consider locally trivial  principal $PU(\cH)$-bundles over a
paracompact  Hausdorff topological space $X$,    the structure group $PU(\cH)$ is equipped with the norm topology.  
The projective unitary group $PU(\cH)$ with the norm topology (Cf. \cite{Kui})  has the homotopy type of an 
Eilenberg-MacLane space $K(\ZZ, 2)$. 
The  classifying space of $PU(\cH)$, as  a classifying space of principal $PU(\cH)$-bundle,  is a $K(\ZZ, 3)$. 
Thus, the set of isomorphism classes of 
principal $PU(\cH)$-bundles over   $X$ is canonically identified
with  (Proposition 2.1 in \cite{AS1})
$$[X, K(\ZZ, 3)] \cong H^3(X, \ZZ).$$ 

  A twisting of  complex $K$-theory on $X$ is given by   a  continuous map $\a: X\to K(\ZZ, 3)$. For such a twisting,
   we can associate a  canonical principal $K(\ZZ, 2)$-bundle
 $\cP_\a$ through  the following pull-back construction
\ba\label{bundle}
\xymatrix{\cP_\a  \ar[d]   \ar[r]  &
EK(\ZZ, 2) 
  \ar[d]  \\
X \ar[r]_\a  & K(\ZZ, 3). } 
\na
 Let $\KK$ be the 0-th space of the complex K-spectrum,  in this paper, 
we   take $\KK$ to be ${\bf Fred} (\cH)$, the space of Fredholm operators on $\cH$.     There is a base-point preserving action
of $K(\ZZ, 2)$ on the K-theory spectrum
\[
K(\ZZ, 2) \times \KK \longrightarrow  \KK
\]
which is represented by the action of complex line bundles on ordinary K-groups.  As we identify  $K(\ZZ, 2) $ with $ PU(\cH)$ and $\KK$ with ${\bf Fred}(\cH)$, the above base point preserving action is given by
the conjugation action
\ba\label{action}
PU(\cH) \times  \Fred (\cH) \longrightarrow \Fred (\cH).
\na 

The action (\ref{action})  defines an associated bundle of  K-theory spectra over $X$. Denote 
\[
\cP_\a (\KK) = \cP_\a\times_{K(\ZZ, 2)} \KK 
\]
the bundle of based spectra over $X$ with fiber the  K-theory spectrum,  and $\{ \Omega^n_X \cP_\a(\KK) = 
 \cP_\a\times_{K(\ZZ, 2)} \Omega^n \KK  \}$  the   fiber-wise iterated loop spaces.

 \begin{definition} The twisted K-groups  of $(X, \a)$ are defined to be
\[
K^{-n}(X, \a) := \pi_0\bigl( C_c(X, \Omega^n_X \cP_\a(\KK))\bigr),
\]
the  set of homotopy classes of compactly supported sections of the bundle of K-spectra.
 \end{definition}

Due to Bott periodicity, we only have two different  twisted K-groups $K^0(X, \a)$ and
$K^1(X, \a)$. 
Given  a closed subspace $A$  of $X$, then $(X, A)$ is   a pair of topological spaces, and
  we define relative 
 twisted K-groups  to be 
 \[
 K^{ev/odd}(X, A;  \a) := K^{ev/odd}(X-A, \a).
 \]
 
\begin{remark} Given a pair of twistings $\a_0, \a_1: X \to K(\ZZ, 3)$, if $\eta: X\times [1, 0] \to K(\ZZ, 3)$ is a homotopy 
between  $\a_0$ and $\a_1$, written as  in the following formation,
\[
  \xymatrix{X  \ar@/^2pc/[rr]_{\quad}^{\a_0}="1"
\ar@/_2pc/[rr]_{\a_1}="2"
&& K(\ZZ, 3),
\ar@{}"1";"2"|(.2){\,}="7"
\ar@{}"1";"2"|(.8){\,}="8"
\ar@{==>}"7" ;"8"^{\eta} }
  \]
 then there is a canonical isomorphism 
$\cP_{\a_0} \cong \cP_{\a_1}$  induced by $\eta$.  This canonical isomorphism
determines a canonical  isomorphism on twisted K-groups 
\ba\label{iso:eta}\xymatrix{ 
\eta_*:  \  K^{ev/odd}(X,  \a_0) \ar[r]^{\cong}  & K^{ev/odd}(X,  \a_1),}
\na
This isomorphism $\eta_*$  depends only on the homotopy class of $\eta$.     
The set of  homotopy classes between  $\a_0$ and $\a_1$ is labelled by 
$[X, K(\ZZ,   2)]$. Note that the the first Chern class  isomorphism
\[
{\bf Vect}_1(X) \cong [X, K(\ZZ,   2)] \cong H^2(X, \ZZ)
\]
where ${\bf Vect}_1(X)$ is 
 the set of equivalence classes of complex line bundles on $X$.  We remark that the isomorphisms 
induced by  two different homotopies 
between  $\a_0$ and $\a_1$   are related through an action of complex line bundles. 
This observation will  play   an important role in the  local index theorem for twisted K-theory.
 \end{remark}

\begin{remark}   Let $\cK $ be the $C^*$-algebra  of  compact 
  operators on $\cH$.  
The isomorphism $PU(\cH) \cong Aut ( \cK)$ via the conjugation action of the unitary group
$U(\cH)$ provides an action of 
$K(\ZZ, 2)$ on the $C^*$-algebra  $\cK$. Hence, any
$K(\ZZ, 2)$-principal bundle $\cP_\a$ defines a locally trivial  bundle of compact operators,
denoted by 
$
\cP_\a(\cK) = \cP_\a\times_{K(\ZZ, 2)} \cK.
$
Let  $C_c(X, \cP_\a(\cK))$ be  the $C^*$-algebra of the compact supported sections of 
$\cP_\a(\cK)$.   We remark that $C_c(X, \cP_\a(\cK)$ is  the (unique up to isomorphism) stable separable complex  continuous-trace $C^*$-algebra over $X$ with its
Dixmier-Douday class $[\a] \in H^3(X, \ZZ)$, here we identify the \v{C}ech cohomology of $X$ with its
singular cohomology (Cf, \cite{Ros} and \cite{Par1}). 
In \cite{AS1} and \cite{Ros},   it was proved that twisted K-groups  $K^{ev/odd}(X, \a)$
are canonically isomorphic to  algebraic K-groups    of the stable   continuous  trace
 $C^*$-algebra  $C_c(X, \cP_\a(\cK))$
\ba\label{twited:K:equ}
K^{ev/odd}(X, \a) \cong KK^{ev/odd} \bigl(\CC, C_c(X, \cP_\a(\cK))\bigr).
\na
\end{remark}

The twisted K-theory is a 2-periodic  {\em generalized cohomology theory}: 
 a contravariant functor on the category
 of pairs consisting   a pair of  topological spaces $A\subset X$with   a twisting $\a: X\to K(\ZZ, 3)$ to the category of
 $\ZZ_2$-graded abelian groups.    Note that a  morphism  between two pairs
 $(X, \a)$ and $(Y, \b)$ is a continuous map $f: X\to Y$ such that $ \b  \circ f =\a$. 
 The twisted K-theory satisfies  the following three axioms whose proofs are rather standard 
 for   2-periodic generalized cohomology theory.  
  \begin{enumerate}

  \item[(I)] ({\bf The homotopy axiom}) If two morphisms  $f, g: (Y, B)\to (X, A)$ are homotopic through
a map $\eta: (Y\times [0, 1] , B \times [0, 1]) \to (X, A)$, written  in terms of  the following homotopy
commutative diagram 
\[
\xymatrix{(Y, B) \ar[d]_{g} \ar[r]^{f} &
(X, A)
 \ar@2{-->}[dl]_{\eta} \ar[d]^{\a} \\
(X, A) \ar[r]_\a  & K(\ZZ, 3), } 
\]
then we have the following commutative diagram
\[
\xymatrix{
& K^{ev/odd}(X, A; \a)\ar[dr]^{g^*} \ar[dl]^{f^*} &\\
K^{ev/odd}(Y, B; \a\circ f)\ar[rr]^{\eta_*} && K^{ev/odd}(Y, B; \a\circ g).}
\]
Here $\eta_*$ is the canonical isomorphism induced
by the homotopy $\eta$.
\item[(II)]   ({\bf The exact  axiom}) For any pair $(X, A)$ with a twisting $\a: X\to K(\ZZ, 3)$, there exists the following
 six-term exact sequence
 \[\xymatrix{
 K^0 (X, A; \a) \ar[r]& K^0  (X, \a)\ar[r]& K^0 (A,\a |_A)
  \ar[d]\\
K^1 (A, \a |_A)  \ar[u] & K^1  ( X, \a) \ar[l]&
    K^1  (X, A; \a) \ar[l]
   }
 \]
here $\a |_A$ is the composition of the inclusion and $\a$.

  \item[(III)]  ({\bf The excision axiom})  Let $(X, A)$ be a pair of spaces and let 
$U \subset A$ be a subspace  such that  the closure $\overline{U} $ is contained in the interior of
 $A$. Then the inclusion $\i: (X-U, A-U) \to (X, A)$ induces, for all $\a: X\to K(\ZZ, 3)$,  an 
isomorphism 
\[
K^{ev/odd}  (X, A; \a)  \longrightarrow K^{ev/odd}  (X-U, A-U; \a\circ \i).
\]
\end{enumerate}
In addition, twisted K-theory satisfies the following basic properties  (see \cite{AS1} \cite{CW2} for detailed proofs).
 \begin{enumerate}
 
  \item[(IV)]  ({\bf Multiplicative property}) Let $\a, \b: X\to K(\ZZ, 3)$ be a pair of twistings on $X$. Denote
 by $ \a+ \b$ the new twisting defined by the following map
 \ba\label{product}
\xymatrix{
\a   +\b :  \quad   X\ar[r]^{(\a   , \b )  \ \  } & K(\ZZ, 3) \times K(\ZZ, 3) \ar[r]^{
\qquad m}   & 
K(\ZZ, 3), }
\na
where $m$ is defined as follows
\[
B PU(\cH) \times BPU(\cH) \cong  B( PU(\cH)  \times  PU(\cH) )    \longrightarrow B PU(\cH), 
\]
for a fixed  isomorphism $\cH \otimes \cH \cong \cH$. 
   Then there is a canonical
multiplication
\[
K^{ev/odd} (X, \a)  \times K^{ev/odd} (X,   \b) \longrightarrow K^{ev/odd} (X, \a + \b ),
\]
which defines a $K^0(X)$-module structure  on   twisted K-groups $K^{ev/odd} (X, \a)$.

 \item[(V)]   ({\bf Thom isomorphism})  Let $\pi: E\to X$ be an oriented real vector bundle of rank $k$  over $X$ with the 
 classifying map
  denoted by $\nu_E: X \to \BSO (k)$, then there is a canonical isomorphism, for any twisting $\a: X\to K(\ZZ, 3)$,
 \ba\label{Thom:CW}
 K^{ev/odd}(X, \a +( W_3 \circ \nu_E) )  \cong K^{ev/odd} (E, \a \circ  \pi),
 \na
with the grading shifted by $k (mod \ 2 )$. 

\nin 
Here  $W_3: \BSO (k) \to  K(\ZZ, 3)$ is   the classifying map of the principal 
$K(\ZZ, 2)$-bundle $\BSpin^c (k) \to \BSO(k) $. 

\item[(VI)]   ({\bf The push-forward map}) 
  For any differentiable map $f: X\to Y$ between two smooth manifolds $X$ and $Y$,  let $\a: Y \to K(\ZZ, 3)$
  be a twisting. Then there is a canonical  push-forward  homomorphism 
  \ba\label{push:forward}
  f_!:   \qquad  K^{ev/odd} \bigl(X, ( \a \circ f)  + (W_3\circ \nu_f) \bigr) \longrightarrow K^{ev/odd}(Y, \a),
  \na
  with the grading shifted  by $n \ mod (2)$ for $n= dim (X)+ dim (Y)$. Here  $\nu_f$ is the classifying map
  $$X \longrightarrow  \BSO(n) $$ corresponding to  the bundle $TX\oplus f^* TY$ over $X$.  
  
  \item[(VII)]  ({\bf  Mayer-Vietoris    sequence})
  If $X$ is covered by two  open subsets $U_1$ and $U_2$ with a twisting  $\a: X\to K(\ZZ, 3)$,  then
    there is a Mayer-Vietoris  exact sequence
 \[
 \xymatrix{
 K^0 (X, \a)\ar[r]& K^1  (U_1\cap U_2, \a_{12})\ar[r]& K^1 (U_1,\a_1 )\oplus K^1( U_2,\a_2 )
  \ar[d]\\
 K^0  (U_1, \a_1) \oplus K^0(U_2,\a_2 )\ar[u] & K^0  (U_1\cap U_2,\a_{12} ) \ar[l]&
   K^1 (X, \a) \ar[l]
   }
 \]
   where $\a_1$, $\a_2$ and $\a_{12}$ are the restrictions of $\a$ to
   $U_1$, $U_2$ and $U_1\cap U_2$ respectively.
      \end{enumerate}

\begin{remark} \begin{enumerate}
\item Note that $[\a +( W_3 \circ \nu_E)] = [\a] +W_3(E)$,   our Thom
isomorphism agrees with the Thom isomorphism in \cite{CW2} and \cite{DK}, where the
notation 
\[
K^{ev/odd}(X, [\a]+W_3 (E) )  \cong K^{ev/odd} (E, \pi^*([\a])) 
 \] 
 is used. 
   
\item The push-forward map  constructed in \cite{CW2}  is established    in the following  form
\[
 f_!:   \qquad  K^{ev/odd} \bigl(X, f^*[\a]+ W_3(TX  \oplus  f^*TY  ) \bigr) \longrightarrow K^{ev/odd}(Y, [\a]), 
 \]
 which is  obtained by applying the   Thom isomorphism and Bott
periodicity as follows. 

Choose an embedding $i:  X \to \RR^{2k}$ . Then $x\mapsto  (f(x), i(x))$  defines  an embedding
 of $X \to Y\times \RR^{2k} $
whose normal bundle $N$ is identified with 
a tubular neighborhood of $X$.  Let  $\nu_N: X\to BSO$   be the classifying map of the normal
bundle $N$, let  $\i: N\to   Y \times \RR^{2k}$  be  the inclusion map, and $\pi:
 Y\times \RR^{2k} \to Y$ be  the projection.   We use the following commutative diagram  
 \[
 \xymatrix{
 N\ar[r]^\i \ar[d] & Y \times \RR^{2k}  \ar[d]^\pi& \\
  X \ar[ur]^{\ (f, i)} \ar[r]^f & Y \ar[r]^\a &K(\ZZ, 3),  }
  \]
 to illustrate induced  twistings $\a\circ f, \a\circ \pi \circ \i$ and $ \a\circ \pi$  on  $X, N$
 and $Y\times \RR^{2nk}$ respectively. 
 Notice  the isomorphism, as  bundles over $X$, 
\[
 N \oplus TX \oplus TX  \cong TX  \oplus  f^*TY \oplus   \underline{\RR}^{2n} 
 \]
 and the canonical $Spin^c$ structure on $TX \oplus TX$  determines   a canonical 
 homotopy between $W_3\circ \nu_N$ and    $W_3\circ \nu_f$, which in turn induces a  canonical  isomorphism \[
K^{ev/odd} \bigl(X, ( \a \circ f)  + (W_3\circ \nu_f) \bigr)  \cong  K^{ev/odd} \bigl(X, ( \a \circ f)  + (W_3\circ \nu_N) \bigr).
\]
 Applying the  Thom isomorphism (\ref{Thom:CW}), we have
\[ 
K^{ev/odd} \bigl(X, ( \a \circ f)  + (W_3\circ \nu_N) \bigr)  \cong K^{ev/odd}  (N, \a \circ \pi \circ \i),
\]
with the grading shifted  by $n \ mod (2)$ for $n= dim (X)+ dim (Y)$.
The  inclusion map $\i: N \to Y\times \RR^{2k}$     induces a natural push-forward map
\[
\i_!: 
K^{ev/odd}  (N, \a \circ \pi \circ \i) \to K^{ev/odd}  (Y \times \RR^{2n}, \a \circ \pi  ).
\]
The Bott periodicity gives a canonical isomorphism
\[ 
K^{ev/odd}  (Y \times \RR^{2n}, \a \circ \pi  ) \cong K^{ev/odd}(Y, \a).
\]
The composition of the above  isomorphisms   and the map $\i_!$
 gives rise to the canonical  push-forward map (\ref{push:forward}).
\end{enumerate} 
\end{remark}

\section{Twisted $Spin^c$-manifolds and analytical index}\label{section:3}

\begin{definition} \label{twisted:cob}  Let $(X, \a)$ be  a paracompact Hausdorff topological space  with a twisting 
$\a$.   An $\a$-twisted $Spin^c$   manifold   over $X$ is  a  quadruple $(M, \nu,  \i, \eta)$ 
where 
\begin{enumerate}
\item $M$ is a smooth,  oriented and {\bf compact}  manifold    together  with a fixed classifying map of its  stable normal bundle  
\[
\nu:   M \longrightarrow  \BSO
\]
here $\BSO= \disp{ \varinjlim_{k } } \BSO(k)$ is  the classifying space of stable normal bundle of $M$;
\item  $\i: M\to X$ is a continuous map;
\item  $\eta$ is an $\a$-twisted $Spin^c$  structure on $M$, that is  a homotopy commutative diagram
\[
\xymatrix{M \ar[d]_{\i} \ar[r]^{\nu} &
\BSO
 \ar@2{-->}[dl]_{\eta} \ar[d]^{W_3} \\
X \ar[r]_\a  & K(\ZZ, 3), } 
\]
where $W_3$ is the classifying map of the principal 
$K(\ZZ, 2)$-bundle $\BSpin^c \to \BSO$  associated to
the third  integral  Stiefel-Whitney class and   $\eta$ is a homotopy
between $W_3 \circ \nu $ and $\a \circ \i$.  
\end{enumerate}
 Two  $\a$-twisted $Spin^c$  structures  $\eta$ and
$\eta'$ on $M$ are called equivalent if there is a homotopy between $\eta$ and $\eta'$.
\end{definition}  

\begin{remark}
\begin{enumerate}
\item Definition    of  twisted $Spin^c$   manifolds   over $X$  is previously  given by Douglas  in \cite{Dou}  using Hopkins-Singer's differential cochains developed in \cite{HopSin}.   Here in Definition 
\ref{twisted:cob}, we define an $\a$-twisted $Spin^c$  structure on $M$ to be a homotopy
between $W_3 \circ \nu $ and $\a \circ \i$ as it induces a canonical isomorphism $\eta_*$ (\ref{2})
which will play an important role in our definition of the analytical index.
\item Let  $(W, \nu, \i, \eta)$ be  an $\a$-twisted $Spin^c$  manifold with boundary over $X$, then there
is a natural $\a$-twisted $Spin^c$   structure on the boundary $\pa W$ with outer normal 
orientation, which is the restriction of the $\a$-twisted $Spin^c$   structure on $W$:
\ba\label{bound}
\xymatrix{\pa W \ar[d]_{\i|_{\pa W} } \ar[r]^{\nu|_{\pa W}} &
\BSO
 \ar@2{-->}[dl]_{\eta|_{\pa W}} \ar[d]^{W_3} \\
X \ar[r]_\a  & K(\ZZ, 3).} 
\na 
\item Given an oriented  real vector bundle $E$ of rank $k$ over a smooth manifold $M$, the classifying map of $E$
\[
\nu_E: M \longrightarrow \BSO(k)
\]
and the principal $K(\ZZ, 2)$-bundle $\BSpin^c(k)  \to \BSO(k)$ define  an associated  twisting
\[
W_3\circ \nu_E:  M \longrightarrow \BSO(k) \longrightarrow K(\ZZ, 3).
\]
\end{enumerate}
\end{remark}

\begin{proposition} \label{condition}
Given a smooth,  oriented and compact  n-dimensional manifold $M$ and a paracompact space $X$ with a   twisting
$\a: X\to K(\ZZ, 3) $, then 
\begin{enumerate}  \item 
$M$ admits an $\a$-twisted $Spin^c$ structure  if and only if  there exists
a continuous map
$\i: M\to X$ such that 
\ba\label{cond}
\i^*([\a]) +  W_3(M)=0
\na
in $H^3(M, \ZZ)$. 
Here $W_3(M)$ is the third integral Stiefel-Whitney class, 
\[
W_3(M) = \b (w_2(M))\]
with $\b: H^2(M, \ZZ_2)\to H^3(M, \ZZ)$   the Bockstein homomorphism and
$w_2(M)$  the second Stiefel-Whitney class of $TM$.  (The condition (\ref{cond}) is the Freed-Witten anomaly cancellation
condition for Type II D-branes (Cf. \cite{FreWit}).)
\item If $\i^*([\a]) + W_3(M) =0$, then the  set  of equivalence classes of $\a$-twisted $Spin^c$ structures
on $M$ are in one-to-one correspondence with elements in $H^2(M, \ZZ)$.
\end{enumerate}
\end{proposition}
\begin{proof}  If $M$ admits an  $\a$-twisted $Spin^c$ structure,  then
$W_3\circ \nu $ and $\a \circ \i$ are homotopic  as maps from $M$ to $K(\ZZ, 3)$. This  means
that  the third integral Stiefel-Whitney class of the stable normal bundle is equal to
$\i^*([\a])$. As $M$ is compact, we can find  
an embedding  
 \[
 i_k: M^n  \longrightarrow \RR^{n+k} 
 \]
for a sufficiently large $k$.   Denote by $\nu(i_k)$ the normal of $i_k$, then we
know that $W_3(\nu(i_k)) = \i^*([\a])$, and 
\[
\nu(i_k)  \oplus TM  \cong  i_k^* (T \RR^{n+k} )
\]
is a trivial bundle, which implies $W_3(M) + W_3(\nu(i_k)) =0$.   Thus, $ \i^*([\a]) +  W_3(M)=0$.

Conversely, if $ \i^*([\a]) +  W_3(M)=0$, then $W_3(\nu(i_k))$ agrees
with $\i^*([\a])$, hence, the classifying map $\nu_k: M\to \BSO(k)$ makes the following
diagram  homotopy commutative  for some homotopy $\eta$:
\[
\xymatrix{M \ar[d]_{\i} \ar[r]^{\nu_k} & 
\BSO(k)
 \ar@2{-->}[dl]_{\eta} \ar[d]^{W_3} \\
X \ar[r]_\a  & K(\ZZ, 3). } 
\]
This defines  an $\a$-twisted $Spin^c$ structure  on $M$ by letting $k\to\infty$.

The set of equivalence classes of $\a$-twisted $Spin^c$ structures on $M$ 
corresponds to   the set of homotopy classes of homotopies between 
$W_3\circ \nu $ and $\a \circ \i$. The latter is an affine space over 
\[
[\Sigma M, K(\ZZ, 3) ] \cong [M, K(\ZZ, 2)]
\]
here $\Sigma$ denotes the suspension.   As $[M, K(\ZZ, 2)]\cong H^2(M, \ZZ)$, so
$ H^2(M, \ZZ) $ acts freely and transitively on the
set of equivalence classes of $\a$-twisted $Spin^c$ structures on $M$.  
\end{proof}

\begin{remark} \begin{enumerate}
\item If the twisting $\a: X\to K(\ZZ, 3)$  is homotopic to the trivial map, then an 
$\a$-twisted $Spin^c$ structure  on $M$  is equivalent to a $Spin^c$ structure  on $M$.
\item Let  $\tau_X: X \to \BSO$  be a classifying  map of the stable tangent bundle 
of $X$, then a $W_3\circ \tau_X$-twisted $Spin^c$ structure  on $M$ is equivalent to
a K-oriented map from $M$ to $X$.
\item  Let $(M, \nu,  \i, \eta)$ be an $\a$-twisted $Spin^c$   manifold   over $X$. Any
K-oriented map $f: M' \to M$ defines a canonical   $\a$-twisted $Spin^c$  structure on $M'$.
\end{enumerate}
\end{remark}

  Recall that
for   $k\in \{0, 1, 2, \cdots\}$ and    a separable $C^*$-algebra $A$,   Kasparov's   K-homology group 
 \[
 KK^{ k} (A, \CC) \cong KK(A, \CL(\CC^k)) 
 \]
  is the abelian group generated by 
 unitary equivalence classes of $\CL(\CC^k)$-graded Fredholm modules over $A$ modulo certain relations (See
 \cite{HigRoe} for details). Then  $KK^{ev}(A, \CC)$ and $KK^{odd}(A, \CC)$ denote the direct
 limits under the periodicity maps
 \[
 KK^{ev} (A, \CC) = \disp{ \varinjlim_{k } } KK^{ 2k}(A, \CC), \ \  \text{and} \ \ 
 KK^{odd} (A, \CC) = \disp{ \varinjlim_{k } } KK^{ 2k+1 }(A, \CC).
 \] 

\begin{definition} Let $X$ be a paracompact Hausdorff space with a twisting $\a: X\to K(\ZZ, 3)$. Let $\cP_\a(\cK)$
be the associated bundle of compact operators on $X$. Analytical twisted K-homology, denoted by 
$K^a_{ev/odd}(X,  \a)$, is defined to be
\[
K^a_{ev/odd} (X, \a) := KK^{ev/odd}\bigl  (C_c(X, \cP_\a(\cK)), \CC \bigr), 
\]
 the 
Kasparov's $\ZZ_2$-graded K-homology of
the   $C^*$-algebra $C_c(X, \cP_\a(\cK))$.   Given a closed subspace $A$ of $X$, the relative 
twisted K-homology $K^a_{ev/odd} (X,  A; \a) $ is defined to be
\[KK^{ev/odd}\bigl  (C_c(X-A, \cP_\a(\cK)), \CC\bigr). \]
 Analytical twisted K-homology is
a 2-periodic generalized homology theory.  
\end{definition}

 We first discuss the relationship
 between the stable normal bundle of $M$ and its stable tangent bundle, and apply it to study
 the corresponding twisted K-homology groups. Note that the
 classifying space of $SO(k)$ is given by the direct limit
 \[
 \BSO(k) =  \disp{ \varinjlim_{m } } Gr(k, m+k)
 \]
  where  $Gr(k, m+k)$ is the Grassmann manifold of oriented $k$-planes in $\RR^{k+m}$.  The classifying 
  space of the stable special orthogonal group is 
    $\disp{ \varinjlim_{k } } \BSO(k)$, and will be denoted by $\BSO$.

   The map 
 $ I_{k, m}: Gr(k, m+k) \longrightarrow Gr(m, k+m)$
  of assigning to each oriented $k$-plane  in $\RR^{k+m}$ to its orthogonal $m$-plane induces
  a map 
  \[
  I: \BSO \longrightarrow \BSO
  \]
  with $I^2$ the identity map. 
   
 For a  compact n-dimensional manifold $M^n$, 
  the stable normal bundle is represented by  the normal bundle
 of an embedding  
$ i_k: M^n  \longrightarrow \RR^{n+k} $
for any sufficiently large $k$. The  normal bundle $\nu(i_k)$ of $i_k$ is the quotient of the pull-back
of the tangent bundle $T \RR^{n+k}  = \RR^{n+k}  \times \RR^{n+k} $ by
the tangent bundle $TM$.    Then the  normal map 
\[
\nu_k: M  \longrightarrow Gr(k, k+n)
\]
 and the tangent map  
 \[
\tau_k: M  \longrightarrow Gr(n, k+n)
\]
are related to each other  by $\tau_k = I_{k, n}\circ \nu_k$. So 
 the  classifying map for the stable normal bundle
 \[
 \nu:   M\longrightarrow  \BSO 
 \]
 and  the  classifying map of the stable tangent bundle
\[
\tau: M\longrightarrow  \BSO
\]
are related by $\tau = I\circ \nu$.  Thus, we have a  natural isomorphism on the associated
bundles of compact operators
\ba\label{tau:nu}
I^*: \tau^*\BSpin^c (\cK) \longrightarrow \nu^* \BSpin^c (\cK).
\na
This determines an isomorphism,  denoted
by $I_*$,   on the 
corresponding twisted K-homology groups
\ba\label{1}
I_*: K^a_{ev/odd} (M, W_3 \circ \tau ) \cong K^a_{ev/odd}(M,     W_3 \circ \nu  ) .
\na
   
   \begin{remark} Given an embedding $i_k: M \to \RR^{n+k}$ with the normal bundle $N$ , the natural isomorphism
   \[
   TM \oplus N \oplus N \cong \underline{\RR}^{n+k} \oplus N
   \]
   and the canonical $Spin^c$ structure on $N \oplus N$ define a canonical homotopy between 
   $W_3\circ \tau$ and $W_3 \circ \nu$. The isomorphism (\ref{1}) is induced by this  canonical homotopy.
 \end{remark}    
    
 For a  Riemannian  manifold   $M$,  denote by  $\CL (TM)$ the bundle of complex Clifford algebras of $TM$ over $M$.  As algebras of the sections,   $C(M, \CL (TM)) $    is Morita equivalent to 
 $C(M, \tau^* \BSpin^c (\cK) )$. Hence, we have   a canonical isomorphism
 \[ 
 K^a_{ev/odd} (M, W_3\circ \tau ) \cong  KK^{ev/odd} ( C(M, \CL (M) ), \CC) 
 \]
 with the degree shift by $dim M  (mod \ 2)$. 
Applying  Kasparov's Poincar\'{e} duality  (Cf. \cite{Kas1}) 
\[
 KK^{ev/odd} (\CC, C(M)) \cong KK^{ev/odd} (C(M, \CL (M) ), \CC), 
 \]
we obtain  a canonical isomorphism 
 \[
PD:  K^0(M)  \cong K^a_{ev/odd} (M, W_3\circ \tau ), 
\]
with the degree shift by $dim M  (mod \ 2)$. 
 The fundamental class   $[M] \in K^a_{ev/odd} (M, W_3\circ \tau )$ is 
the Poincar\'{e} dual of the unit element in $K^0(M)$. Note that
 $[M]\in K^a_{ev}(M, W_3(M))$ if $M$ is even dimensional and $[M]\in K^a_{odd}(M, W_3(M))$ if $M$ is odd dimensional. The cap product 
 \[\cap: 
 K^a_{ev/odd} (M, W_3\circ \tau ) \otimes K^0(M)  \longrightarrow K^a_{ev/odd} (M, W_3\circ \tau )  
 \]
 is  defined by the Kasparov product. We remark that the cap product of the fundamental class $[M]$ and
 $[E] \in K^0(M)$ is given by 
 \[
 [M]\cap [E] = PD ([E]).
 \]

  Given an    $\a$-twisted $Spin^c$  manifold $(M, \nu,  \i, \eta)$   over $X$,  the homotopy $\eta$ induces an isomorphism
$\nu^* \BSpin^c   \cong   \i^* \cP_\a $ as principal $K(\ZZ, 2)$-bundles  on $M$, hence defines   an isomorphism 
\[\xymatrix{
   \nu^* \BSpin^c (\cK)   \ar[rr]^{  \eta^*}_{ \cong }&&  \i^* \cP_\a  (\cK) }
\]
as bundles of $C^*$-algebras  on  $M$.    This isomorphism
determines a canonical isomorphism 
between the corresponding continuous trace $C^*$-algebras 
\[
C(M, \nu^* \BSpin^c (\cK))  \cong  C(M,  \i^* \cP_\a(\cK) ).
\]
 Therefore, we have a canonical isomorphism
\ba\label{2}
\eta_*:   K^a_{ev/odd} (M, W_3   \circ \nu    )  \cong   K^a_{ev/odd} (M,      \a \circ \i  ). 
\na
Notice that the natural push-forward map in analytic K-homology theory is 
 \ba\label{3}
 \i_*: K^a_{ev/odd} ( M, \a \circ \i ) \longrightarrow 
 K^a_{ev/odd}( X,  \a ).
 \na
We can introduce a notion of analytical index for any $\a$-twisted $Spin^c$   manifold   over $X$, taking values in
analytical twisted K-homology of $(X, \a)$.

\begin{definition}\label{index:analytical} Given an $\a$-twisted $Spin^c$ {\bf closed}  manifold $(M, \nu, \i, \eta)$ and $[E]\in K^0(M)$,
 we define its
analytical index  
\[
 \ind_a((M, \nu, \i, \eta), [E]) \in K^a_{ev/odd} (X, \a)
\]
 to be the image of the cap product  $[M]\cap[E] \in K^a_{ev/odd} (M, W_3\circ \tau )$    under the  maps (\ref{1}), (\ref{2}) and (\ref{3}):
\[\xymatrix{
  K^a_{ev/odd} (M, W_3\circ \tau ) \ 
 \ar[rr]^{ \qquad  \i_* \circ \eta_* \circ I_*} && K^a_{ev/odd} (X, \a).}
 \]
 \end{definition}
 
 The analytical index enjoys the following properties.

\begin{proposition}
\label{a-index:pro} 
\begin{enumerate}
\item The analytical index  $\ind_a((M, \nu, \i, \eta) ,  [E]) $ depends only on the equivalence
class of the $\a$-twisted $Spin^c$ structure $\eta$.   
\item (Disjoint union and Direct sum) Given a pair of $\a$-twisted $Spin^c$  manifolds
$(M_1, \nu_1, \i_1, \eta_1)$ and $ (M_2, \nu_2, \i_2, \eta_2)$, and $[E_i] \in K^0(M_i)$. Then
\[\begin{array}{lll}
&&  \ind_a \bigl( (M_2, \nu_2, \i_2,   \eta_2 ) \sqcup (M_2, \nu_2, \i_2,   \eta_2 ) , [E_1] \sqcup [E_2] \bigr)  \\[2mm]
&=&   \ind_a((M_1, \nu_1, \i_1,   \eta_1 ) , [E_1])  + 
 \ind_a((M_2, \nu_2, \i_2,   \eta_2 ) , [E_2]) .
 \end{array}
\]
\item (Bordism invariance)  If $(W, \nu, \i, \eta)$ is an $\a$-twisted $Spin^c$  manifold with boundary over $X$ and $[E]\in K^0(W)$, then
\[
\ind_a (\pa W, \pa  \nu , \pa \i , \pa \eta ),   [E|_{\pa W} ] )  =0.
\]
\end{enumerate}
\end{proposition}
\begin{proof}  \  The $\a$-twisted $Spin^c$ structure $\eta$ enters the definition of   $\ind_a((M, \nu, \i, \eta),  [E]) $
only through 
\[
  \eta_*:   K^a_{ev/odd} (M, W_3   \circ \nu    )  \cong   K^a_{ev/odd} (M,      \a \circ \i  ).
  \]
  This isomorphism  depends  only on the homotopy class of $\eta$. So  Claim (1) is obvious.

Claim (2) follows from the disjoint union and direct sum property of the fundamental classes and the cap product.
 
To establish Claim (3), let $(W, \nu, \i, \eta)$ be  an $\a$-twisted $Spin^c$  manifold with boundary over $X$, and denote
its boundary by $M=\pa W$ with the induced  $\a$-twisted $Spin^c$ structure $(\pa \nu, \pa \i, \pa \eta)$. 
Let  $i: M \to W$  be  the boundary inclusion map.   The exact sequence  in topological  K-theory
and analytical   K-homology.   are related  through Poincar\'{e} duality  (assume  that $W$ is odd dimensional) as in the following
  commutative diagram
\ba\label{commu:0}
\xymatrix{
K^1(W, M) \ar[d]^{PD} & K^0(M) \ar[l]\ar[d]^{PD} &  K^0(W) \ar[l]^{i^*} \ar[d]^{PD}\\
K^a_0(W,  W_3 \circ \tau_W)   & K^a_0(M, W_3\circ \tau_M) \ar[l]^{i_*}   &  K^a_1(W,M; W_3\circ \tau_W) \ar[l]^{\pa}, }
\na
where $\tau_M$ and $\tau_W$ are classifying maps of the stable tangent bundles of $M$ and
$W$ respectively.   One could get this
K-homology  exact sequence by applying the Kasparov  KK-functor to the following
short exact sequence of $C^*$-algebras:
\[
0\to C_0(W, \CL(W))\longrightarrow C(W, \CL(W))  \longrightarrow C( M,  \CL(M))\to 0,
\]
where $C_0(W, \CL(W))$ denotes the $C^*$-algebra of  continuous sections of $\CL(W)$ vanishing at the boundary
$M$. The relative    analytical   K-homology  $KK^{ev/odd} (C_0(W, \CL(W)) , \CC)$  is 
isomorphic to $K^a_{ev/odd} (W, M; W_3\circ \tau )$, and hence isomorphic to  $K^a_{ev/odd} (W, M; W_3\circ \nu )$ under
(\ref{1}). 
 
   From (\ref{commu:0})  that the Poincar\'{e} dual  of   $$ [i^* E] \in K^0(M) $$
is mapped to zero in $K^a_0(W,  W_3\circ \tau_W)$ for $[E]\in K^0(W)$ under the map $i_*$:
\ba\label{id:0}
i_*\circ PD \circ i^* ([E]) = i_* \circ \pa \circ PD ([E]) =0.
\na

Notice that  $\ind_a \bigl( (M,  \pa \nu , \pa \i, \pa \eta), [E|_M]  \bigr)$
is image of  the   class $PD (i^*[E])$ under  the following sequence of maps
\[\xymatrix{
K^a_0(M, W_3\circ \tau_M) \ar[r]^{I_*} & K^a_0(M, W_3\circ \pa \nu)\ar[r]^{(\pa \eta)_*} &   K^a_0(M , \a\circ \pa \i)
\ar[r]^{(\pa \i)_*} & K^a_0(X, \a), } 
\]
and the inclusion map  $i: M \to W$ induces the following commutative diagram
\[
\xymatrix{ K^a_0(M, W_3\circ \tau_M) \ar[d]^{i_*}  \ar[r]^{I_*}  &  K^a_0(M, W_3\circ \pa \nu)\ar[d]^{i_*} \ar[r]^{(\pa \eta)_*}
& K^a_0(M, \a \circ \pa \i) \ar[d]^{i_*} \ar[dr]^{(\pa \i)_*}\\
K^a_0(W,  W_3\circ \tau_W)  \ar[r]^{I_*}    &  K^a_0(W, W_3\circ   \nu)  \ar[r]^{\eta_*}  &  
K^a_0(W, \a \circ \nu) \ar[r]^{\i_*}  &  
K^a_0(X, \a). } 
\]
Therefore, we conclude that
\[
\begin{array}{lllll}
&& \ind_a  \bigl( (M,  \pa \nu , \pa \i, \pa \eta), [E|_M] \bigr) &\qquad&\\[2mm]
&=& (\pa \i)_* \circ (\pa \eta)_* \circ I_* \circ PD (i^*[E]) &\qquad&\text{(Definition \ref{index:analytical}) } \\[2mm]
&=&  \i_*\circ \eta_* \circ I_* \circ i_*  \circ PD (i^*[E])  &\qquad&\text{(The above commutative diagram)} \\[2mm]
&=&  0  &\qquad&\text{(\ref{id:0})} \\[2mm]
\end{array}
\]

\end{proof}

\begin{remark} 
Given an $\a$-twisted $Spin^c$ structure $\eta$ on $(M, \nu, \i)$ and a complex line bundle
$L$ over $M$, denote  by $c_1 \cdot [\eta]  $ the action of the first Chern class $ c_1 = c_1(L)
\in H^2(M, \ZZ)$ on  the homotopy class of $\eta$, then the analytical index depends
on the choice of equivalence classes of $\a$-twisted $Spin^c$ structures through the following
formulae
 \[
\ind_a((M, \nu, \i, c_1\cdot [\eta]) , [E]) = \ind_a((M, \nu, \i,  [\eta]) ,  ([L]\otimes [E]) ).
\]
\end{remark}

 \section{Twisted $Spin^c$  bordism and topological index}  \label{section:4}

Given a manifold $X$ with a twisting $\a: X\to K(\ZZ, 3)$, $\a$-twisted $Spin^c$  manifolds over $X$  form a  bordism category,
 called the $\a$-twisted $Spin^c$    brodism over $(X, \a)$,  whose 
 objects are compact smooth  manifolds   over $X$ with an $\a$-twisted $Spin^c$ structure
 as in Definition \ref{twisted:cob}. 
 A morphism between   $\a$-twisted $Spin^c$  manifolds
$(M_1, \nu_1, \i_1, \eta_1)$ and $ (M_2, \nu_2, \i_2, \eta_2)$  is
a boundary preserving  continuous map $f: M_1 \to M_2$ and  the following diagram 
 \ba\label{morphism}
   \xymatrix{
    M_1  \ar@/{}_{1pc}/[ddr]_{\i_1} \ar@/{}^{1pc}/[drr]^{\nu_1}
       \ar[dr]^{f }            \\
      & M_2 \ar[d]_{\i_2} \ar[r]^{\nu_2}   
                     & \BSO \ar@2{-->}[dl]_{\eta_2}\ar[d]^{W_3}       \\
      & X \ar[r]_\a   & K(\ZZ, 3)              }
   \na
 is a homotopy commutative diagram such that
 \begin{enumerate}
\item   $\nu_1$ is homotopic to $\nu_2 \circ f$ through a continuous  map $\nu: M_1 \times [0, 1] \to \BSO$;
\item  $\i_2 \circ f$ is homotopic to $\i_1$   through continuous   map $\i : M_1 \times [0, 1] \to X$;
\item   the composition of homotopies  $( \a \circ \i ) * (\eta_2 \circ (f\times Id) ) * (W_3 \circ \nu)$
 is homotopic to $\eta_1$.
   \end{enumerate} 
   The boundary functor $\pa$ applied to an $\a$-twisted $Spin^c$  manifold  $(M, \nu, \i, \eta)$ is the
   manifold  $\pa M$  with outer normal orientation  and  the restriction of  the $\a$-twisted $Spin^c$ 
   structure to $M$. 
    
   Two  $\a$-twisted $Spin^c$  manifolds
$(M_1, \nu_1, \i_1, \eta_1)$ and $ (M_2, \nu_2, \i_2, \eta_2)$  are called isomorphic
if  there exists a diffeomorphism $f: M_1 \to M_2$ such that  the diagram (\ref{morphism}) is
a homotopy commutative diagram.
  
\begin{definition} We say that an $\a$-twisted $Spin^c$  manifold $(M, \nu, \i, \eta)$ is null-bordant
if there exists an $\a$-twisted $Spin^c$  manifold  $W$    whose
boundary is $ (M, \nu, \i, \eta)$ in the sense of (\ref{bound}). 
We define the  $\a$-twisted $Spin^c$  bordism  group of $X$, denoted by $\Omega^{Spin^c}(X, \a)$,
to be the set of all  isomorphism classes of   closed  $\a$-twisted $Spin^c$  manifolds over $X$ modulo
null-bordism, with the sum given by the disjoint union.
 \end{definition}

The subgroup of  isomorphism classes of  n-dimensional  closed  $\a$-twisted $Spin^c$  manifolds over $X$
will be denoted  $\Omega^{Spin^c}_n(X, \a)$. Define
\[\begin{array}{c}
 \Omega^{Spin^c}_{ev} (X, \a)=  \bigoplus_k  \Omega^{Spin^c}_{2k}(X, \a)\\[2mm]
  \Omega^{Spin^c}_{odd} (X, \a) = \bigoplus_k  \Omega^{Spin^c}_{2k+1}(X, \a).
  \end{array}
  \]

\begin{proposition}The analytical index defined in the previous section  induces a homomorphism
\ba\label{a-index}
\ind_a:  \Omega^{Spin^c}_{ev/odd} (X, \a)  \longrightarrow K^a_{ev/odd} (X, \a)
\na 
 \end{proposition}
 \begin{proof} Let $(M, \i, \nu, \eta)$ be  $\a$-twisted $Spin^c$  manifold  over $X$, representing an
 element in the  $\a$-twisted $Spin^c$    bordism group $ \Omega^{Spin^c}_{ev/odd}  (X, \a)$. Define
 \[
 \ind_a(M, \i, \nu, \eta)  = \ind_a( (M, \i, \nu, \eta), [\underline{\CC}]) \in K^a_{ev/odd} (X, \a),
 \]
 where $\underline{\CC}$ denotes the trivial line bundle over $M$, representing the
 unit element in $K(M)$. We need to show that for a pair of isomorphic  objects 
 \[(M_1, \i_1, \nu_1, \eta_1) \  \text{and }\  (M_2, \i_2, \nu_2, \eta_2)
 \]
 in the  $\a$-twisted $Spin^c$  bordism  category over $X$,  we have
 \[
 \ind_a(M_1, \i_1, \nu_1, \eta_1)= \ind_a  (M_2, \i_2, \nu_2, \eta_2).
 \]
 Let $f$ be  a diffeomorphism  from $M_1$ to $M_2$ such that   (\ref{morphism}) is
 a homotopy commutative diagram.  
 Let $\tau_1$ and $\tau_2$ be classifying maps of the stable tangent bundles of $M_1$ and $M_2$
respectively.  
  The homotopy between $\nu_1$ and $\nu_2\circ f$ implies that $\tau_1$ 
  and $\tau_2\circ f$ are homotopy equivalent.  This defines a canonical $Spin^c$
  structure on  $T M_1 \oplus f^* TM_2$. Hence, there is a canonical    Morita
equivalence 
\[
C(M_1, \CL(M_1))  \sim C(M_1 , f^* \CL(M_2 )). 
\]
This Morita equivalence  defines a canonical isomorphism
\[
K^a_{ev/odd} (M_1, W_3\circ \tau_1) \cong K^a_{ev/odd} (M_1, W_3\circ \tau_2\circ f ).  
\]
Recall that 
natural  push-forward map in analytical K-homology is related to the    K-theoretical 
push-forward map $f_!$  in topological K-theory via the
Poincar\'{e} duality (PD):
\[
\xymatrix{
K^{ev/odd} (M_1) \ar[d]^{\cong}_{PD} \ar[r]^{f_!} &  K^{ev/odd} (M_2) \ar[d]^{\cong}_{PD}  \\
K^a_{ev/odd} (M_1, W_3\circ \tau_1) \ar[r]^{f_*} & K^a_{ev/odd} (M_2, W_3\circ \tau_2) }
\]
where the Poincar\'{e} duality  shifts the degree by the dimension of the underlying manifold. Applying    the natural  push-forward map in analytical K-homology, we obtain
\[
f_*: K^a_{ev/odd} (M_1, W_3\circ \tau_1 ) \cong K^a_{ev/odd} (M_1,  W_3\circ \tau_2\circ f  )  \longrightarrow K^a_{ev/odd} (M_2, W_3\circ \tau_2)
\]
with the degree shifted by $d(f) = dim M_1-dim M_2 ( mod\  2 ).$ 
The homotopy  between $W_3\circ \nu_1$ and  $W_3 \circ \nu_2 \circ f$ 
 defines a canonical homomorphism 
 \[
 f_*: K^a_{ev/odd} (M_1,  W_3\circ \nu_1 )  \cong  K^a_{ev/odd} (M_1,  W_3 \circ \nu_2 \circ f) 
\longrightarrow K^a_{ev/odd} (M_2, W_3\circ \nu_2), 
 \]
 such that the following diagram commutes
 \ba\label{com:1}
 \xymatrix{
 K^a_{ev/odd} (M_1, W_3\circ \tau_1 ) \ar[d]^{f_*}\ \ar[r]^{I^{M_1}_*}  & K^a_{ev/odd} (M_1,  W_3\circ \nu_1 )\ar[d]^{f_*}\\
  K^a_{ev/odd} (M_2, W_3\circ \tau_2 )  \ar[r]^{I^{M_2}_*}  & K^a_{ev/odd} (M_2,  W_3\circ \nu_2 ).
 }
 \na
 Similarly, the homotopy between 
 $
 ( \a \circ \i ) * (\eta_2 \circ (f\times Id) ) * (W_3 \circ \nu)
$ and $\eta_1$  induces a
 commutative diagram
  \ba\label{com:2}
 \xymatrix{
 K^a_{ev/odd} (M_1, W_3\circ \nu_1 ) \ar[d]^{f_*}\ \ar[r]^{(\eta_1)_*}  & K^a_{ev/odd} (M_1,\a\circ \i_1   )\ar[d]^{f_*}\\
  K^a_{ev/odd} (M_2,  W_3\circ \nu_2  )  \ar[r]^{(\eta_2)_*}  & K^a_{ev/odd} (M_2,  \a\circ \i_2).
 }
 \na
 The homotopy   between  $\a  \circ \i_2 \circ f$    and $\a \circ \i_1$   induces the following commutative
  triangle
\ba\label{com:3}
 \xymatrix{
 K^a_{ev/odd} (M_1,\a\circ \i_1   )\ar[dd]^{f_*} \ar[dr] ^{(\i_1)_*}&\\
& K^a_{ev/odd} (X, \a) .  \\
 K^a_{ev/odd} (M_2 , a  \circ \i_2 ) \ar[ur]^{(\i_2)_*}&}
 \na
 These commutative diagrams (\ref{com:1}),  (\ref{com:2})  and  (\ref{com:3}) imply that
  \[\begin{array}{lll}
  && \ind_a (M_2, \i_2, \nu_2, \eta_2) 
 \\[2mm]
 &=& (\i_2)_* \circ (\eta_2)_*\circ I^{M_2}_* ([M_2]) \\[2mm]
 &=&  (\i_2)_* \circ (\eta_2)_*\circ I^{M_2}_*\circ f_* ([M_1])\\[2mm]
 &=& (\i_1)_* \circ (\eta_1)_*\circ I^{M_1}_* ([M_1]) \\[2mm]
 &= &  \ind_a(M_1, \i_1, \nu_1, \eta_1).
 \end{array}
 \]
 Now the cobordant invariance in  Proposition \ref{a-index:pro} tells us that $\ind_a$ is a  well-defined homomorphism 
 from 
 $ \Omega^{Spin^c}_{ev/odd} (X, \a)$ to $ K^a_{ev/odd} (X, \a)$.
  \end{proof}

We recall the construction of Thom spectrum of $\bf Spin^c$ bordism.  Let  $\xi_k$  be  the universal bundle over $\BSO(k)$. The pull-back 
 bundle over  $\BSpin^c(k)$ is  given by
 \[
\tilde{\xi}_k =  ESpin^c(k) \times _{Spin^c(k)}  \RR^k. 
 \]
 Denote by $\MSpin (k)$ the Thom space of  $ \tilde{\xi}_k$. 
 The inclusion map $j_k$ induces a pull-back diagram
 \[
 \xymatrix{ j_k^* \tilde{\xi}_{k+1} \ar[r]  \ar[d] & \tilde\xi_{k+1} \ar[d]  \\
  \BSpin^c(k) \ar[r]^{j_k} &   \BSpin^c(k+1)}
  \]
  with   $j_k^* \tilde{\xi}_{k+1} \cong 
\tilde{\xi}_k \oplus \underline{\RR}$, where $\underline{\RR}$ denote
 the trivial  real  line bundle. Then  the Thom space of $ j_k^* \tilde{\xi}_{k+1} $ 
 can be identified with $\Sigma \MSpin (k)$ (the suspension of $\MSpin (k)$). Thus
 we have  a sequence of  continuous maps 
 \[
 \Th (j_k): \Sigma \MSpin (k) \longrightarrow \MSpin (k+1),
 \]
 that means, $\{ \MSpin (k) \}_k$ is the Thom spectrum associated to $\BSpin^c =\disp{\vlim_k}\BSpin^c(k) $.

 As  $ \BSpin^c(k) $   is a  principal 
   $K(\ZZ, 2)$-bundle  over $ \BSO(k) $, we have a base point preserving
   action of $K(\ZZ, 2)$ on the Thom spectrum  $\{ \MSpin (k) \}$,  written as 
   \[
   K(\ZZ, 2)_+ \wedge   \MSpin (k)  = \dfrac{K(\ZZ, 2)   \times   \MSpin (k) }{ K(\ZZ, 2) \times *}\longrightarrow   \MSpin (k)
   \]
   which is compatible with the base-point action of $K(\ZZ, 2)$ on the K-theory spectrum $\KK$
   in the sense that       
    there exists a $K(\ZZ, 2)$-equivariant map, called 
    the index  map
    \[
    Index: \MSpin   \longrightarrow  \KK. 
\]
  This  $K(\ZZ, 2)$-equivariant map has been  constructed in \cite{Dou} and \cite{Wald}. Here we 
  provide   a  more geometric  construction. Write the principal $BU(1)$-bundle $\BSpin^c(2k)$ as the following
  pull-back  bundle
  \[
  \xymatrix{
\BSpin^c (2k)   \ar[r]   \ar[d]   &  EK(\ZZ, 2)\ar[d] \\
 \BSO(2k)  \ar[r]^{W_3} & K(\ZZ, 3),}
\]
which induces a natural $PU(\cH)$-action
\[
PU(\cH) \times \BSpin^c(2k)  \longrightarrow \BSpin^c (2k).
\]
This   action corresponds to the action of  the set of complex line bundles on the
set of $Spin^c$ structures.  The $PU(\cH)$-action 
on $\BSpin^c (2k)$ can be lifted
to a base point preserving action of $PU(\cH)$ on $\MSpin (2k)$
\[
PU(\cH) \times \MSpin (2k)  \longrightarrow \MSpin  (2k).
\]  Note that there is
a  fundamental  $\ZZ_2$-graded spinor bundle   $S^+ \oplus S^-$ over  $\BSpin^c (2k)$
see Theorem C.9  in \cite{BM}, which defines a canonical Thom class in
$K^0(\MSpin (2k))$. 
This canonical Thom class  determines  a $PU(\cH)$-equivariant map 
 \[
Index:   \MSpin (2k)  \longrightarrow  \Fred (\cH).
\] 
   Hence we have  associated 
   bundles of Thom spectra  over $X$
      \[
   \cP_\a ( \MSpin (k)) = \cP_\a\times_{K(\ZZ, 2)} \MSpin (k),
   \]
   and  natural maps  to the associated bundle of K-theory spectra
   \[
 Index:   \cP_\a ( \MSpin (k))  \longrightarrow   \cP_\a ( \KK ) = \cP_\a\times_{K(\ZZ, 2)}  \KK.
 \]
 
    \begin{remark}  The
 $Spin^c$  bordism groups over  a pointed space $X$,  denoted by $\Omega_*^{Spin^c}(X)$ as  in  
 (\cite{Stong}),   can be identified as the stable homotopy groups of $\MSpin  \wedge X$ (the 
  Pontrjagin-Thom  isomorphism)
   \ba\label{Pont-Thom}
 \Omega_n^{Spin^c}(X) \cong \pi^S_n(\MSpin \wedge X) := \disp{\vlim_{k} } \pi_{n+k} (\MSpin(k) \wedge X).
 \na
 The index map $Index: \MSpin \to \KK$  determines a natural
 transformation from the even or odd dimensional $Spin^c$  bordism group of $X$ to K-homology of $X$, which is called  the topological index:
 \[
 \Omega_{ev/odd}^{Spin^c}(X) \to K^t_{ev/odd} (X).
 \]
 \end{remark}
 The following theorem   is the twisted version
 of the Pontrjagin-Thom isomorphism (\ref{Pont-Thom}).
 
 \begin{theorem}\label{P-Thom}  The bordism group $\Omega^{Spin^c}_n (X, \a)$ of  n-dimensional $\a$-twisted $Spin^c$  manifolds  over $X$ is isomorphic to the stable homotopy group
 \[\xymatrix{
\Theta:  \Omega^{Spin^c}_n (X, \a) \ar[rr]^{\cong\qquad } && \pi_n^{\bf S}\bigl(\cP_\a ( \MSpin  ) /X\bigr).}
\]
Here we denote $\pi_n^{\bf S}\bigl(\cP_\a ( \MSpin ) /X\bigr) : 
= \disp{\lim_{k\to\infty}} \pi_{n+k} \bigl( \cP_\a ( \MSpin (k)) /X\bigr)$.
\end{theorem}
\begin{proof}    The proof is modeled on the proof of the classical Pontrjagin-Thom isomorphism 
(Cf. \cite{Stong})

 \nin{\bf Step 1.} Definition of the homomorphism $\Theta$.
 
 \vspace{2mm}
 
Let $\sigma$ be an element in $\Omega^{Spin^c}_n (X, \a) $ represented by a $n$-dimensional 
$\a$-twisted $Spin^c$  manifold $(M, \i, \nu, \eta)$   over $X$. Let $i_k: M \to \RR^{n+k}$
be an embedding with  the classifying map  of  the normal bundle denoted by $\nu_k$. Then
we have the following pull-back diagram
\ba\label{normal}
\xymatrix{
N  \ar[r]^{\tilde \nu_k}\ar[d]_{\pi}  & \xi_k\ar[d] \\
 M \ar[r]^{\nu_k} & \BSO(k).}
\na
Here the total space $N$ of the   normal bundle of $i_k$ can be thought
of as a subspace of $\RR^{n+k} \times \RR^{n+k}$.  Under the addition 
map $\RR^{n+k} \times \RR^{n+k}\to \RR^{n+k}$, 
for some sufficiently small $\eps>0$, the $\eps$-neighborhood $N_\eps$ 
of the zero section $M\times \{0\}$ of $N$ is an embedding  $\eps|_{N_\eps}: N_\eps \to \RR^{n+k} $,
whose restriction to  the zero section $M\times \{0\}$ is the embedding $i_k: M \to \RR^{n+k}$.

Consider $S^{n+k}$ as $\RR^{n+k}\cup\{\infty\}$ (the one point compactification) so we have an
embedding $N_\eps \to S^{n+k}$.  Define
$$
c: S^{n+r}  \to N_\eps/\pa N_\eps $$
 by collapsing all points  of $S^{n+k}$ outside and on the boundary
of $N_\eps$ to a point.  Note that $N_\eps/\pa N_\eps$ is homeomorphic to the Thom
space 
$\Th (N)$ of the normal bundle of $i_k$, induced by multiplication by $\eps^{-1}$.  
 Denote  this homeomorphism by 
\[
\eps^{-1}: N_\eps/\pa N_\eps \longrightarrow  \Th (N).
\]

  The pull-back diagram 
\[
\xymatrix{\cP_{W_3\circ \nu_k}     \ar[d]   \ar[r]  &
EK(\ZZ, 2) 
  \ar[d]  \\
M \ar[r]_{W_3\circ \nu_k}    & K(\ZZ, 3), } 
\]
 induces a homotopy pull-back  
\[
\xymatrix{\cP_{W_3\circ \nu_k}   (\BSpin^c (k))    \ar[d]  \ar[r]  &
EK(\ZZ, 2) (\BSpin^c(k)  )
  \ar[d]  \\
M \ar[r]_{W_3\circ \nu_k}    & K(\ZZ, 3). } 
\]
As $EK(\ZZ, 2)$ is contractible, so $EK(\ZZ, 2) (\BSpin^c(k)  )$ is homotopy equivalent to
$\BSO(k) $.  This implies  that the following diagram
\[
\xymatrix{\cP_{W_3\circ \nu_k}   (\BSpin^c (k) )   \ar[d]  \ar[r]  &
\BSO (k) 
  \ar[d]  \\
M \ar[r]_{W_3\circ \nu_k}    & K(\ZZ, 3)  } 
\]
is a homotopy pull-back.  Notice that the diagram
\[
\xymatrix{M  \ar[d]^{Id}    \ar[r]^{\nu_k} &
BSO (k) 
  \ar[d]^{W_3}   \\
M \ar[r]_{W_3\circ \nu_k}   & K(\ZZ, 3) } 
\]
is commutative.  Thus there exists a unique map (up to
homotopy) $h: M \to \cP_{W_3\circ \nu_k}   (\BSpin^c (k) ) $  such
that the following diagram
\[
 \xymatrix{
    M  \ar@/{}_{1pc}/[ddr]_{Id}  \ar@/{}^{1pc}/[drr]^{\nu_k}
       \ar[dr]^{h }            \\
      &   \cP_{W_3\circ \nu_k}   (\BSpin^c (k) )   \ar[d]  \ar[r]    
                     & \BSO (k) \ar@2{-->}[dl]  \ar[d]^{W_3}       \\
      & M \ar[r]_{W_3\circ \nu_k}     & K(\ZZ, 3).            }
      \]
is  homotopy commutative.
Together with  the pull-back diagram 
\[
\xymatrix{ 
\tilde{\xi}_k \ar[r] \ar[d] & \xi_k \ar[d]\\
\BSpin^c (k)   \ar[r]   & \BSO (k), } 
\]
we obtain a homotopy commutative diagram 
\[
\xymatrix{N  \ar[d]     \ar[r]  &
  \cP_{W_3\circ \nu_k}   (\tilde{\xi}_k )
  \ar[d]    \\
M \ar[r]    &    \cP_{W_3\circ \nu_k}   (\BSpin^c (k) )  } 
\]
which in turn determines 
a  canonical map
\[
h_*:  \Th (N) \longrightarrow \cP_{W_3\circ \nu_k}  (\MSpin (k))/M.
\]
Notice that the $\a$-twisted $Spin^c$-structure on $M$ defines   a continuous map 
\[
  \i_* : \cP_{W_3\circ \nu_k}   (\MSpin (k))/M  \longrightarrow  \cP_\a (  \MSpin (k)  )/X .
\]
The  composition  $ \i_* \circ h_* \circ \eps^{-1}\circ c $ is a continuous map of pairs
\[
\theta= \theta_{(M, \i, \nu, \eta)}:  (S^{n+k}, \infty) \longrightarrow (\cP_\a (  \MSpin (k)  )/X, *),
\]
here   $*$ is the base point   in $\cP_\a (  \MSpin (k) )/X$, hence represents an element 
of  $$\disp{\lim_{k\to\infty}} \pi_{n+k} \bigl( \cP_\a ( \MSpin (k)) /X\bigr).$$

The stable  homotopy class of $\theta$
doesn't depend on   choices  of  $i_k,  \nu_k, \eps$  in the construction for sufficiently large $k$.  Thus, we assign an element  in 
$$\disp{\lim_{k\to\infty}} \pi_{n+k} \bigl( \cP_\a ( \MSpin (k)) /X\bigr)$$ represented by $\theta$ to every
closed $\a$-twisted $Spin^c$  manifold $(M, \i, \nu, \eta)$   over $X$.  

Now we  show  that
the stable  homotopy class of $\theta$ depends only on the bordism class of $M$. Let
$W$ be an (n+1)-dimensional   $\a$-twisted $Spin^c$   manifold and 
$j:  \pa W \to \RR^{n+k}$ be  an embedding for some sufficiently large $k$ with the classifying map
$\nu_k$ for the  normal  bundle
\[
\xymatrix{
N_{\pa W}   \ar[r]^{\tilde \nu_k}\ar[d]_{\pi}  & \xi_k\ar[d] \\
\pa W \ar[r]^{\nu_k} & \BSO(k).}
\]
 Choose $i: W \to \RR^{n+k} \times [0, 1]$ to  be an embedding  agreeing with $j\times \{1\}$ on $\pa W$,  embedding
a tubular neighborhood of $\pa W$ orthogonally  along $j(\pa W) \times \{1\}$, and with  the image missing
$ \RR^{n+k}  \times \{0\}$.  The previous
construction applied to  the embedding $i$ yields a null-homotopy of the map
$$\theta: (S^{n+k}, \infty) \longrightarrow  (  \cP_\a   (\MSpin (k)  )/X, *).$$

Assigning the
stable homotopy class of the map $\theta$ to each $\a$-twisted $Spin^c$  bordism class, we have defined a map
\[
\Theta:  \Omega^{Spin^c}_n (X, \a) \longrightarrow \pi^{\bf S}_{n} \bigl( \cP_\a ( \MSpin) /X\bigr).
\]

 \vspace{4mm}
 
\nin{\bf Step 2.} $\Theta$ is a  homomorphism.
\vspace{2mm}

Given a pair of closed $\a$-twisted $Spin^c$  manifolds $(M_1, \i_1,  \nu_1, \eta_1)$
and  $(M_2, \i_2,  \nu_2, \eta_2)$  over $X$ representing two classes in $ \Omega^{Spin^c}_n (X, \a)$. 
Then for $a=1$ or $2$,  $\Theta([M_a, \i_a, \nu_a, \eta_a])$ is represented by a map
\[
\theta_a: (S^{n+k}, \infty)  \longrightarrow  (\cP_\a   (\MSpin (k)  )/X, *)
\]
  constructed as above.

Choose an  embedding  $i: M_1\sqcup M_2  \to \RR^{n+k}$ such that  the last coordinate is positive for $M_1$ and negative for $M_2$.
Taking small enough $\epsilon$, the previous construction  gives us a map 
\[
\xymatrix{
(S^{n+k}, \infty)  \ar[r]^d & S^{n+k} \wedge  S^{n+k} \ar[r]^{\qquad \theta_1 + \theta_2\qquad \qquad  } &(  \cP_\a  (\MSpin (k)  )/X, *),
}
\]
where $d$ denotes the collapsing the equator of $S^{n+k}$. This map represents
the sum of the homotopy classes of $\theta_1$ and $\theta_2$.
Hence, 
\[
 \Theta ([M_1, \i_1,  \nu_1, \eta_1]   +  [M_2, \i_2,  \nu_2, \eta_2] ) = 
  \Theta ([M_1, \i_1,  \nu_1, \eta_1] ) + 
  \Theta ( [M_2, \i_2,  \nu_2, \eta_2] ).
  \]
  
  \nin{\bf Step 3.} $\Theta$ is a monomorphism.
  \vspace{2mm}

  Let  $(M , \i ,  \nu, \eta)$ be an  $\a$-twisted $Spin^c$  $n$-manifold such that $\Theta ([M , \i ,  \nu, \eta])=0$. Then
  for some large $k$, the above  construction in Step 1  defines a continuous map 
  \[
  \theta_0 = \i_* \circ h_* \circ \eps^{-1}\circ c: ( S^{n+k}, \infty) \longrightarrow \cP_\a(\MSpin (k))/X
  \]
  which is null-homotopic.  As    $\cP_\a(\BSpin^c(k)) \subset  \cP_\a(\MSpin (k))/X$ and   the fact that $M$ is the zero section of $N$, and the map 
  \[
  \i_*\circ h_*: \Th (N) \to  \cP_\a(\MSpin (k))/X
  \]
  sending the zero section of $N$  to $\cP_\a(\BSpin^c(k))$. We have 
    \[
  M= \theta_0^{ -1} \bigl(\cP_\a(\BSpin^c(k))\bigr).
  \]
  Note that  the trivial map, denoted by
  $\theta_1$,  maps  $ S^{n+k}$   to the base point of $\cP_\a(\MSpin (k))/X$,
  so we know that
  $\theta_1^{-1}  \bigl(\cP_\a(\BSpin^c(k))\bigr)$ is a   empty set.   Now we can choose a homotopy 
  between  $\theta_0$ and $\theta_1$ 
  \[
  H: S^{n+k} \times [0, 1] \longrightarrow  \cP_\a(\MSpin (k))/X
  \]
  for some sufficiently large $k$, such that $H$ is differentiable near and transversal to
 $$\cP_\a(\BSpin^c(k)) \subset  \cP_\a(\MSpin (k))/X.$$
  So 
 $$W= H^{-1}\bigl (\cP_\a(\BSpin^c(k))\bigr)$$
  is a submanifold of $\RR^{n+k}\times [0, 1]$
  with $\pa W= M$ meeting $\RR^{n+k}\times \{0\}$ orthogonally along $M$. The map 
 $H|_W$ sends $W$ to $\cP_\a(\BSpin^c(k))$, as  $\cP_\a(\BSpin^c(k))$ is a fibration over $X$, so
 we have a continuous map $\i_W: W \to X$. 

Note that the pull-back diagram 
\[
\xymatrix{\cP_\a  \ar[d]   \ar[r]  &
EK(\ZZ, 2) 
  \ar[d]  \\
X \ar[r]_\a  & K(\ZZ, 3), } 
\]
 induces a homotopy pull-back  
\[
\xymatrix{\cP_\a (\BSpin^c(k) )    \ar[d]  \ar[r]  &
EK(\ZZ, 2) (\BSpin^c k) )
  \ar[d]  \\
X \ar[r]_\a  & K(\ZZ, 3). } 
\]
As $EK(\ZZ, 2)$ is contractible, so the associated fiber bundle
 $EK(\ZZ, 2) (\BSpin^c(k) )$ is homotopy equivalent to
$\BSO(k) $. This implies  that the following diagram
\[
\xymatrix{\cP_\a (\BSpin^c(k))   \ar[d]  \ar[r]  &
\BSO(k)
  \ar[d]  \\
X \ar[r]_\a  & K(\ZZ, 3)  } 
\]
is a homotopy pull-back. 
We see that the map $H|_W$ defines a homotopy commutative
diagram
\[
 \xymatrix{
    W   \ar@/{}_{1pc}/[ddr]  \ar@/{}^{1pc}/[drr] 
       \ar[dr]^{H|_W }            \\
      &   \cP_\a (\BSpin^c (k) )   \ar[d]  \ar[r]    
                     & \BSO  (k) \ar@2{-->}[dl]  \ar[d]^{W_3}       \\
      & X \ar[r]_\a   & K(\ZZ, 3).            }
      \]
  Hence,  $W$ admits  an  $\a$-twisted $Spin^c$ structure such that the boundary inclusion
  $    M   \to W$ is a morphism in the $\a$-twisted $Spin^c$  bordism category. 
  This implies that $(M , \i ,  \nu, \eta)$
  is null-bordant, so $[M , \i ,  \nu, \eta]=0$ in  $\Omega^{Spin^c}_n (X, \a)$.
       
 \vspace{4mm}
 
\nin{\bf Step 4.}  $\Theta$ is an epimorphism.
\vspace{2mm}

Let $\theta : ( S^{n+k}, \infty) \longrightarrow (\cP_\a(\MSpin (k))/X, *)$, for a large $k$,  represents an element
in $$\pi_n^{\bf S}\bigl(\cP_\a ( \MSpin  ) /X\bigr).$$
As $S^{n+k}$ is compact, we may find a finite
dimensional model for $\BSpin^c(k)$, so we may pretend that $\BSpin^c(k)$
is finite dimensional.  We can deform the map $\theta$ to a map
$h$ such that
\begin{enumerate}
\item $h$ agrees with $\theta$ on an open set containing  $\infty$.
\item $h$ is differentiable on the preimage of some open set   containing $\cP_\a (\BSpin^c(k))$ and is transverse on 
$\cP_\a (\BSpin^c(k))$.  
\item  Let $M= h^{-1} (\cP_\a (\BSpin^c(k)))$, then $h$ is  a normal bundle map from a tubular
neighborhood of $M$ in $S^{n+k}$ to $\cP_\a(\MSpin (k))/X$. 
\end{enumerate}
Then $M$ is a smooth  compact n-dimensional manifold with the following homotopy commutative diagram:
\[
 \xymatrix{
    M   \ar@/{}_{1pc}/[ddr]_\i  \ar@/{}^{1pc}/[drr]^\nu 
       \ar[dr]^{h|_M }            \\
      &   \cP_\a (\BSpin^c (k) )   \ar[d]  \ar[r]    
                     & \BSO (k) \ar@2{-->}[dl]  \ar[d]^{W_3}       \\
      & X \ar[r]_\a   & K(\ZZ, 3).            }
      \]
      Therefore, $M$ admits an  $\a$-twisted $Spin^c$ structure $(\i, \nu, \eta)$. The above generalized
      Pontrjagin-Thom construction implies that
      $\Theta ([M, \i, \nu, \eta] )$ is the class represented by  $\theta$.

\end{proof}

 The index map $Ind: \MSpin  \to \KK$ (the complex K-theory spectrum)
  induces a  map of  bundles of   spectra over $X$
 \[
Ind:   \cP_\a ( \MSpin ) \longrightarrow   \cP_\a ( \KK).
 \]
 The  stable homotopy group of  $ \cP_\a ( \KK)/X$ by definition is the twisted
 topological K-homology groups $K^t_{ev/odd} (X, \a)$, due to the periodicity of $\KK$, we have
 \[
 K^t_{ev} (X, \a) =  \disp{\vlim_{k\to\infty}} \pi_{2k} \bigl( \cP_\a ( \KK) /X\bigr)
 \]
 and 
  \[
 K^t_{odd} (X, \a) =  \disp{\vlim_{k\to\infty}} \pi_{2k+1} \bigl( \cP_\a ( \KK ) /X\bigr).
 \]
 Here the direct limits are taken by the double suspension
 \[
 \pi_{n+2k} \bigl( \cP_\a ( \KK ) /X\bigr) \longrightarrow  \pi_{n+2k+2} \bigl( \cP_\a (S^2 \wedge   \KK ) /X \bigr)
 \]
 and then followed by the standard map 
 \[
 \xymatrix{
 \pi_{n+2k+2} \bigl(    \cP_\a (S^2 \wedge \KK)/X \bigr)  \ar[r]^{b\wedge  1} & \pi_{n+2k+2} \bigl(  \cP_\a (\KK \wedge \KK)/X \bigr)
  \ar[r]^m &  \pi_{n+2k+2} \bigl(  \cP_\a ( \KK)/X\bigr) }, 
 \]
 where $b: \RR^2\to \KK$ represents  the Bott generator  in $K^0(\RR^2)$, $m$ is the base point preserving
 map inducing the ring structure on K-theory.

\begin{definition} \label{index:topological}  (Topological index)  There is a homomorphism, called
the topological index 
\ba\label{t-index}
\ind_t:  \Omega^{Spin^c}_* (X, \a)  \longrightarrow K^t_{ev/odd} (X, \a), 
\na
  defined to be $Ind_*\circ \Theta$,
 the composition of $\Theta$ (as in Theorem \ref{P-Thom}) 
 \[\xymatrix{
\Theta:  \Omega^{Spin^c}_n (X, \a) \ar[rr]^{\cong\qquad } && \pi_n^{\bf S}\bigl(\cP_\a ( \MSpin  ) /X\bigr).}
\]
and the induced index transformation
\[
Ind_*:   \disp{\vlim_{k\to\infty}} \pi_{n+2k} \bigl( \cP_\a ( \MSpin (2k)) /X\bigr) \to 
\disp{\vlim_{k\to\infty}} \pi_{n+2k} \bigl( \cP_\a ( \KK) /X\bigr)
\]
\end{definition}

\section{Topological index = Analytical index} \label{section:5}
 
 In this section,we will establish  the main result of this paper. It should be thought of  as the generalized  Atiyah-Singer index theorem for   $\a$-twisted $Spin^c$ manifolds over $X$ with a twisting $\a: X\to K(\ZZ, 3)$. In this section, we assume that $X$ is a {\bf smooth}  manifold.

\begin{theorem}\label{Main} There is a natural  isomorphism $\Phi: K^t_{ev/odd} (X, \a) \longrightarrow
K^a_{ev/odd} (X, \a)$ such that the following diagram
commutes
\[\xymatrix{ 
& \Omega^{Spin^c}_{ev/odd} (X, \a)  \ar[dl]_{\ind_t} \ar[dr]^{\ind_a}
&\\
K^t_{ev/odd} (X, \a) \ar[rr]^{\cong}_{\Phi} & & K^a_{ev/odd} (X, \a),}
\]
that is, given an closed $\a$-twisted $Spin^c$  manifold $(M, \nu, \i, \eta)$ over $X$,  we have 
\[
\ind_a ( M, \nu, \i, \eta)  = \ind_t (M, \nu, \i, \eta)
\]
under the isomorphism $\Phi$. 
\end{theorem}

 \begin{remark}
 If $\a: X\to K(\ZZ, 3)$ is  the trivial map, then we have following commutative diagram
\ba\label{AS:index}
\xymatrix{ 
& \Omega^{Spin^c}_{ev/odd} (X)  \ar[dl]_{\ind_t} \ar[dr]^{\ind_a}
&\\
K^t_{ev/odd} (X) \ar[rr]^{\cong} & & K^a_{ev/odd} (X)}
\na
where the isomorphism $K^t_{ev/odd} (X) \cong K^a_{ev/odd} (X)$ follows from the work
of Atiyah \cite{Ati2}, Baum-Douglas \cite{BD1} and  Kasparov \cite{Kas}. If $X$ is a point, then
the diagram (\ref{AS:index})  is the usual form of Atiyah-Singer index theorem for $Spin^c$ manifolds. 
 \end{remark}

  \begin{proof} Notice that $K^t_{ev/odd} (X, \a)$ and $K^a_{ev/odd} (X, \a)$ are  two 
   generalized homology theories  dual to the twisted K-theory. 
  The twisted K-cohomology
$K^{ev/odd}(X, \a)$ is defined as  
\[
K^{ev}(X, \a) =  \disp{\vlim_{k\to \infty}}  \pi_0(\Gamma(X, \cP_\a(\Omega^{2k} \KK))  \]
\[
K^{odd}(X, \a) =  \disp{\vlim_{k\to \infty}}  \pi_0(\Gamma(X, \cP_\a(\Omega^{2k+1} \KK))
\]
the homotopy classes of sections (with compact support if $X$ is non-compact) of the associated bundle of  K-theory spectra, and 
$\Omega^k \KK$ is the iterated loop space of $\KK$.   We will show that there are natural isomorphisms
from  twisted K-homology (topological and analytical) to twisted K-cohomology with
the twisting shifted by 
\[
\a \mapsto \a   + ( W_3 \circ \tau)
\]
where $ \tau: X \to BSO $ is 
  the classifying map of  the stable tangent space and   $\a   + ( W_3 \circ \tau )$ denotes  the map
$X \to  K(\ZZ, 3)$  defined in (\ref{product}), representing the class $[\a]+ W_3(X)$ in $H^3(X, \ZZ)$.

 \vspace{4mm}

\nin{\bf Step 1.}  There exists an isomorphism   $K^t_{ev/odd} (X, \a) \cong    K^{ev/odd} (X, \a  + (  W_3 \circ \tau) ) $   with the degree shifted by $dim X (mod\ 2)$. 
 \vspace{2mm}

Assume $X$ is $n$ dimensional, choose an embedding $i_{2k}: X\to \RR^{n+2k}$ for some large
$k$,  with its normal bundle $\pi: N_{2k}  \to X$ identified as an $\eps$-tubular neighborhood of $X$. 
Any two embeddings   $X\to \RR^{n+2k}$ are  homotopic through a regular homotopy  for a sufficiently large $k$.  Under the inclusion $\RR^{n+2k} \times {0} \subset \RR^{n+2k+2}$, the Thom spaces 
of $N_{2k}$ and $N_{2k+2}$ are related through the reduced suspension by $S^2$
\ba
\label{Thom:X}
\Th (N_{2k+2})  = S^2 \wedge  \Th (N_{2k}).
\na 

By the  Thom isomorphism (\cite{CW2}), we have an isomorphism
\ba\label{Thom:isom}
 K^{ev/odd} \bigl(X, \a + (   W_3 \circ \tau)\bigr)  \cong   \disp{\vlim_{k\to \infty}} K^{ev/odd} (N_{2k}, \a\circ \pi ),
 \na
 where  $\a\circ \pi: N_{2k} \to K(\ZZ, 3)$ is the pull-back twisting on $N_{2k} $.  There is a natural map
 from $K^{ev/odd} (N_{2k}, \a\circ \pi )$ to $K^t_{ev/odd}(X, \a)$ by considering $S^{n+2k}$
 as $\RR^{n+2k} \cup \{\infty\}$ and the following pull-back diagram
 \[
 \xymatrix{
 \cP_{\a \circ \pi} (\KK) \ar[r] \ar[d] & \cP_\a(\KK) \ar[d]  \\
 N_{2k}\ar[r]^{\pi}& X.
  }
 \]
  Given an element of $K^{ev} (N_{2k}, \a\circ \pi )$ represented by a compactly supported section $\theta:  
  N_{2k} \to  \cP_{\a\circ \pi}(\KK)$, then  
  \[
  \xymatrix{
  S^{n+2k} \ar[r]^c & \Th(N_{2k}) \ar[r]^{\theta}&   \cP_{\a\circ \pi}(\KK)/N_{2k} \ar[r] &  \cP_\a(\KK)/X}
 \]
 representing an element in $K^t_{ev/odd}(X, \a)$.  Replacing $X$ by $X\times \RR$,   this construction
 gives a map from $K^{odd} (N_{2k}, \a\circ \pi )$ to $K^t_{odd}(X, \a)$.
  Recall that there is  a   homotopy equivalence
 $\KK\sim \Omega^2 \KK$ induced by the  map
 \[
 \xymatrix{
  S^2  \wedge  \KK \ar[r]^{b\wedge  1} &\KK \wedge  \KK \ar[r]^m & \KK}, 
 \]
 where $b$ represents  the Bott generator  in $K^0(\RR^2)$, $m$ is the base point preserving
 map inducing the ring structure on K-theory.  Together with (\ref{Thom:X}), we obtain
 \ba\label{diagram:limit}
 \xymatrix{
  S^{n+2k} \ar[r]^c\ar[d]^s & \Th(N_{2k}) \ar[r]\ar[d]^s &  \cP_\a(\KK)/X\ar[d]^s\\ 
 S^2 \wedge  S^{n+2k} \ar[d] \ar[r]^{1\wedge c \qquad } & S^2 \wedge   \Th(N_{2k})  \ar[d] 
 \ar[r] & \cP_\a( S^2 \wedge \KK)/X \ar[d]^{m\circ (b \wedge  1)} \\
 S^{n+2k+2}  \ar[r]^{c} &   \Th(N_{2k+2})  
 \ar[r] & \cP_\a( \KK)/X
 } 
 \na
 where $s$ is the reduced suspension map by $S^2$.   This implies that 
 the  stable homotopy equivalent class  of  sections
 defines  the same element in $K^t_{*}(X, \a)$, with the degree given by  $n (mod\ 2)$.  
 Thus,  we have a well-defined homomorphism
 \ba\label{Psi}
  \Psi_t:  K^{ev/odd} (X, \a  + (   W_3 \circ \tau) )   \longrightarrow   K^t_{ev/odd}(X, \a).
  \na
  
 Conversely, for a sufficiently large $k$,  let $\theta: (S^{m+2k}, \infty)  \to (\cP_\a (\KK)/X, *)$   represent an element in $K^t_{ev/odd} (X, \a)$ (depending on even or odd $m$). We can lift this map to  a map 
 $\theta_0:  S^{m+2k} \to \cP_\a ( \MSpin (2k)) /X$. As in Step 4 of  the proof of Theorem \ref{P-Thom},   $\theta_0$ can be deformed  to a differentiable map $h$ on the preimage of 
 some open set containing  $\cP_\a ( \BSpin^c (2k)) $,  is transverse to  $\cP_\a ( \BSpin^c (2k))$
 and agrees  with $\theta_0$ on an open set containing $\infty$. Then 
 \[
 M= h^{-1}  (\cP_\a ( \BSpin^c (2k))) \subset \RR^{m+2k} =  S^{m+2k} -\{\infty\}
 \]
 is a smooth compact manifold and 
 admits a natural $\a$-twisted $Spin^c$   structure 
 \ba  \label{recon:M}
 \xymatrix{
    M   \ar@/{}_{1pc}/[ddr]_\i  \ar@/{}^{1pc}/[drr]^\nu
       \ar[dr]^{h|_M }            \\
      &   \cP_\a (\BSpin^c (2k) )   \ar[d]  \ar[r]    
                     & \BSO \ar@2{-->}[dl]  \ar[d]^{W_3}       \\
      & X \ar[r]_\a   & K(\ZZ, 3).            }
      \na
  Therefore, we can assume that the
   map $\theta: (S^{m+2k}, \infty)  \to (\cP_\a (\KK)/X, *)$ comes from   the following 
   commutative diagram
   \[
   \xymatrix{& \RR^{m+2k} \ar[r]^\theta & \cP_\a (\KK)  \ar[dd]& \\
   N_\eps    \ar[ur]^j  \ar[dr]^\pi  && \\
   &M\ar[uu]  \ar[r]^\i & X,} 
   \]
   where $N_\eps$ is the normal bundle of $M$ in $\RR^{n+2k}$, identified as   the $\eps$-neighborhood  of $M$ in $ S^{m+k}$.  
   In particular, the continuous  map 
   $$\theta\circ j: N_\eps \to S^{m+2k} \to  \cP_\a (\KK)/X$$
   determines a compactly supported section of $\cP_{\a \circ \i\circ \pi}(\KK)=( \i\circ \pi)^*\cP_\a (\KK) $.

   Choose an embedding $\i_{2k_0}: M\to \RR^{2k_0}$, 
   the  $\a$-twisted $Spin^c$ structure (\ref{recon:M})  on $M$ over $X$ induces a natural 
   $\a\circ \pi $-twisted $Spin^c$ structure (\ref{recon:M})  on $M$ over $X \times \RR^{2k_0}$
   \[
   \xymatrix{
   M   \ar[d]_{(\i, \i_{2k_0})}   \ar[r]^{\nu}    
                     & \BSO \ar@2{-->}[dl]  \ar[d]^{W_3}       \\
      X  \times \RR^{2k_0} \ar[r]_{\a\circ \pi_0}   & K(\ZZ, 3)           }
      \]
 such that $(\i, \i_{2k_0})$ is an embedding.  Here $\pi_0$ is the projection $X\times \RR^{2k_0}\to X$.
   Notice that 
   \[
   K^t_{ev/odd} (X, \a)  \cong K^t_{ev/odd} (X \times \RR^{2k_0}, \a).
   \]
  Therefore,    
      without losing any generality, we may assume  that $\i: M\to X$ is an embedding and there is
   an embedding $i_{2k}: X\to \RR^{n+2k}$.  Denote by  $N_X$ the normal bundle of  the embedding $i_{2k}$,
   and $N_M$   the normal bundle of $M$ in $ \RR^{n+2k}$.  We implicitly 
   assume that any  normal bundle of an embedding is  identified to a tubular neighborhood
   of the embedding. 
  Then  we have  the following collapsing map
  \[
  \Th(N_X) \longrightarrow \Th(N_M)
  \]
  as $N_M$ is imbedded in $N_X$ with appropriate choices of tubular neighborhood.  
  
  The map $\theta: 
  (S^{m+2k}, \infty)  \to (\cP_\a (\KK)/X, *)$  is stable homotopic to
  \ba\label{theta:X}\xymatrix{
  (S^{n+2k}, \infty)   \ar[r]^{c} &( \Th(N_M), *) \ar[r] & ( \cP_\a (\KK)/X, *)}.
  \na
 Hence,  we obtain a map
  \[\xymatrix{ 
  \Th (N_X)  \ar[r]  &  \Th (N_M ) \ar[r] &  \cP_\a (\KK)/X}, 
  \]
  which gives  a  compactly supported section of $\cP_{\a \circ \pi}  (\KK)$ where  $\pi$ denotes the
  projection $N_X\to X$. This section defines an element  in $K^{ev}(N_X, \a \circ \pi)$, hence an element of 
 $K^{ev} (X, \a +( W_3 \circ \tau))$   under the isomorphism (\ref{Thom:isom})  and the diagram 
 (\ref{diagram:limit}).  It is straightforward to show that this map from $ K^t_{ev/odd} (X, \a)$ to 
 $K^{ev/odd} (X, \a +( W_3 \circ \tau))$ is the inverse of $\Psi_t$  defined before.

  Hence, we have established  the isomorphism
\ba\label{Psi:t}
\Psi_t:
 K^{ev/odd} (X, \a + (W_3 \circ \tau)) \longrightarrow   K^t_{ev/odd} (X, \a), 
 \na
  with the degree shifted by $dim X (mod\ 2)$. 
This  is the Poincar\'{e} duality  in topological twisted K-theory.
 
 \vspace{4mm}

\nin{\bf Step 2}.  There is an isomorphism $ \Psi_a: K^a_{ev/odd} (X, \a) \cong    K^{ev/odd} (X, \a  + (  W_3 \circ \tau) )$   with the degree shifted by $dim X (mod\ 2)$. 

\vspace{2mm}

 Recall that for a twisting $\a:X \to K(\ZZ, 3)$,
  there is an associated bundle of  $C^*$ algebras, denoted by
  $\cP_\a(\cK)$ where $\cK $ be the $C^*$-algebra  of  compact 
  operators on an infinite dimensional, complex and separable Hilbert space $\cH$.  
   Here we identify   $K(\ZZ, 2)$ as the projective unitary group $PU(\cH)$ with the 
  norm topology (See \cite{AS1} for details). 
   There is an equivalent definition of $ K^{ev/odd} (X, \a)$ in \cite{Ros},  
using the continuous trace $C^*$ algebra   $C_c(X, \cP_\a(\cK))$,  which consists of compactly supported sections of the
 bundle of  $C^*$-algebras,  $\cP_{\a } (\cK)$. 
  Moreover,
  Atiyah-Segal established a canonical 
  isomorphism in \cite{AS1} between 
  $ K^{ev/odd} (X, \a )$ and the analytical  K-theory of   $ C_c(X, \cP_\a(\cK))$.
   The latter K-theory  can be described as  the Kasparov KK-theory 
  \[
  KK\bigl(\CC,  C_c(X, \cP_\a(\cK))\bigr). 
  \]

There is an equivalent definition of $ K^{ev/odd} (X, \a  +( W_3 \circ \tau) )$ in \cite{Ros},  
using the continuous trace
$C^*$ algebras,  which consists of compactly supported sections of the
 bundle of  $C^*$-algebra,  $\cP_{\a  +( W_3 \circ \tau)} (\cK)$ 
 
In \cite{EEK} (see also \cite{Tu})  , a natural isomorphism, called the Poincar\'{e} duality  in analytical  twisted K-theory,
\[
 KK\bigl(\CC,  C_c(X, \cP_{\a} (\cK))\hat{\otimes}_{C_c(X) }C_c(X, \CL(TX))\bigr)   \cong   
 KK\bigl( C_c(X, \cP_{-\a} (\cK)), \CC \bigr) 
\]
  is constructed using the  Kasparov product with  the weak dual-Dirac element  associated to 
  $\cP_{\a} (\cK)$, see Definition 1.11 and Theorem 1.13 in \cite{EEK} for details. 
  
  Note that there is a natural Morita equivalence  
    \[
C_c(X, \cP_{\a} (\cK))\hat{\otimes}_{C_c(X) }C_c(X, \CL(TX))
  \sim C_c(X,  \cP_{\a + ( W_3 \circ \tau)} (\cK)) 
  \]
  which induces a canonical isomorphism on their KK-groups. 
 The isomorphism 
 \[
 KK\bigl( C_c(X, \cP_{-\a} (\cK)) , \CC \bigr)  \cong 
  KK\bigl( C_c(X, \cP_{\a} (\cK)), \CC \bigr)
  \]
    is obvious using the  operator conjugation. 
  So  in  our notation, the Poincar\'{e} duality in  analytical twisted K-theory can be written  in the following form 
  \ba\label{Psi:a}
  \Psi_a:
  K^{ev/odd} (X, \a + ( W_3 \circ \tau) )  \longrightarrow    K^a_{ev/odd} (X, \a)  ,
 \na
   with the degree shifted by $dim X (mod\ 2)$ coming from the shift of grading on
   the evev/odd dimensional complex Clifford algebra.
  
 Applying the Poincar\'{e} duality isomorphisms (\ref{Psi:t}) and (\ref{Psi:a})
 in   topological  twisted K-theory and algebraic twisted K-theory, we have a natural isomorphism
  \[\xymatrix{
  \Phi: & 
K^t_{ev/odd} (X, \a) \ar[rr]^{\Psi_a \circ  \Psi_t^{-1} } && K^a_{ev/odd} (X, \a),
}\]
such that the following diagram commutes
\ba\label{twisted:index-0}
\xymatrix{ 
&  K^{ev/odd} (X, \a + ( W_3 \circ \tau   ) )  \ar[dl]_{\Psi_t}^{\cong} \ar[dr]^{\Psi_a}_{\cong }
&\\
K^t_{ev/odd} (X, \a) \ar[rr]^{\cong}_{\Phi} & & K^a_{ev/odd} (X, \a).}
\na
 
 \vspace{4mm}
 
\nin{\bf Step 3.}  Show that $\ind_a = \Phi \circ \ind_t$.

Applying the suspension operation, we only  need to prove the even case.  Let $(M, \i, \nu, \eta)$ be a $2n$ dimensional  closed $\a$-twisted $Spin^c$   manifold  over $X$, 
\[
\xymatrix{M \ar[d]_{\i} \ar[r]^{\nu} &
\BSO
 \ar@2{-->}[dl]_{\eta} \ar[d]^{W_3} \\
X \ar[r]_\a  & K(\ZZ, 3), } 
\]
representing an
 element in the  $\a$-twisted $Spin^c$    bordism group $ \Omega^{Spin^c}_{n}  (X, \a)$. 
 
 The analytical index of $(M, \i, \nu, \eta)$, as defined Definition \ref{index:analytical},   
 is given by 
 \ba\label{i_!}
 \ind_a (M, \i, \nu, \eta) = \i_! ([M]) = \i_! \circ PD ([\underline{\CC}]).
 \na
  where $PD: K^0(M) \to   K^a_0( M, W_3\circ \tau)$ is the Poincar\'{e} duality isomorphism
  with $\tau$ the classifying map for the stable tangent bundle.
  The push-forward map $\i_!$  in (\ref{i_!}) is obtained from  the following sequence  of maps
 \[
 \xymatrix{
 K^a_0(M, W_3\circ \tau ) \ar[r]^{I_*} & K^a_0(M, W_3\circ \nu)  \ar[r]^{\eta_*}  &
 K^a_0(M, \a\circ \i )   \ar[r]^{\i_*}& K^a_0(X,  \a) .
 }
  \]
There is a natural push-forward
 map 
 \[
 \i_!:   \Omega^{Spin^c}_{ev} (M , \a\circ \i )  \to   \Omega^{Spin^c}_{ev} (X , \a)
 \]
 such that the following diagrams  for the analytical index 
 \[
 \xymatrix{
  \Omega^{Spin^c}_{ev} (M , \a\circ \i )\ar[d]_{\ind_a}
   \ar[r]^{\i_!}  & \Omega^{Spin^c}_{ev} (X , \a) \ar[d]^{\ind_a}\\
   K^a_0(M, \a\circ \i ) \ar[r]^{\i_!} & K^a_0(X, \a),}
   \]
   and for the topological index
   \[
 \xymatrix{
  \Omega^{Spin^c}_{ev} (M , \a\circ \i )\ar[d]_{\ind_t}
   \ar[r]^{\i_!}  & \Omega^{Spin^c}_{ev} (X , \a) \ar[d]^{\ind_t}\\
   K^t_0(M, \a\circ \i ) \ar[r]^{\i_!} & K^t_0(X, \a),}
   \]
   are commutative.   
   
   As $(M, Id, \nu, Id)$ is a  natural $\a\circ \i $-twisted $Spin^c$ manifold over $M$, we only need
   to show that
   \ba\label{id:1}
 (I_*)^{-1}  \circ \Phi \circ  \ind_t(M, Id, \nu, Id)  = [M] = PD (\underline{\CC}) 
  \na
  in $K^a_0(M, W_3 \circ \tau)$.  We will show that the identity (\ref{id:1})
   follows from the Thom isomorphism 
   \[
K^0(M) \cong K^0 (N_M, W_3\circ \nu_k \circ \pi)
\]
where  we choose an embedding $i_k: M\to \RR^{2n+k}$ with its normal bundle 
  $N_{M}$,   $\pi$ is  the projection $N_M\to M$ and 
  $\nu_k: M  \to BSO(k)$ is  the classifying map of the normal bundle $N_M$.  The image of
  $[\underline{\CC}] $ under the above Thom isomorphism
  is represented by the map 
  \[
  \theta_M:   (S^{2n+k}, \infty)  \to ( \Th (N_M), *) \to  (\cP_{W_3\circ \nu_k \circ \pi} (\KK) /N_M , *)
  \]
 arising from   the 
  $W_3 \circ \nu_k$-twisted $Spin^c$  structure on $M$ as in the following diagram
  \[
  \xymatrix{
  \Th (N_M) \ar[r] & \cP_{W_3\circ \nu_k} (\MSpin (k)) /M \ar[r] & \cP_{W_3\circ \nu_k} (\KK) /M\\
  M \ar@{_{(}->}[u]  \ar[dr]^{Id}  \ar[r] & \cP_{W_3\circ \nu_k} (\BSpin^{c} (k))\ar[d] \ar@{_{(}->}[u]  \ar[r] & BSO(k)\ar[d]^{W_3} \\
  & M \ar[r]^{W_3 \circ \nu_k} & K(\ZZ, 3).}
  \]
  This same diagram also defines the topological index 
  of $(M, Id, \nu, Id)$  under the index map
  \[
\ind_t:    \Omega^{Spin^c}_{ev} (M , W_3 \circ \nu) \to   K^t_0(M, W_3 \circ \nu).
\]
Hence, we establish  the 
  following commutative diagram
\[\xymatrix{ 
& \Omega^{Spin^c}_{ev/odd} (X, \a)  \ar[dl]_{\ind_t} \ar[dr]^{\ind_a}
&\\
K^t_{ev/odd} (X, \a) \ar[rr]^{\cong}_{\Phi} & & K^a_{ev/odd} (X, \a).}
\]

   \end{proof}

 \begin{remark}   \label{remark:index}   Let  $\pi_X: TX \to X$ be   the projection. Applying  the Thom isomorphism (\cite{CW2}), we  obtain
 the following   isomorphism 
\[
K^{ev/odd} (TX, \a  \circ \pi_X)  \cong K^{ev/odd} ( X, \a + ( W_3 \circ \tau)).  
\]
Hence, the above commutative diagram (\ref{twisted:index-0}) becomes
\ba\label{twisted:index-1}
\xymatrix{ 
&  K^{ev/odd} (TX, \a  \circ \pi_X)  \ar[dl]^{PD}_{\cong} \ar[dr]_{PD}^{\cong }
&\\
K^t_{ev/odd} (X, \a) \ar[rr]^{\cong}_{\Phi} & & K^a_{ev/odd} (X, \a),}
\na
which should  be thought of as a generalized Atiyah-Singer index theorem. 
for $\a$-twisted $Spin^c$  manifolds  over $X$ with a twisting $\a: X\to K(\ZZ, 3)$. 
If $\a: X\to K(\ZZ, 3)$ is a trivial map, the  commutative 
diagram (\ref{twisted:index-1}) becomes
\[
\xymatrix{ 
&  K^{ev/odd} (TX)  \ar[dl]_{\ind_t} \ar[dr]^{\ind_a}
&\\
K^t_{ev/odd} (X) \ar[rr]^{\cong}_{\Phi} & & K^a_{ev/odd} (X),}
\]
which is the basic form for the Atiyah-Singer index theorem. The upper vertex 
represents the symbols of elliptic pseudo-differential operators on $X$.
 Each of these index maps is essentially 
just the Poincar\'{e} duality isomorphism between the K-cohomology of
$T^*X$ and  the two realizations of  the K-homology $K_0(X)$.  See \cite{BD1} for more details.

 In particular, if $\a$ is the twisting associated to the classifying 
map $W_3\circ \tau: X \to K(\ZZ, 3)$ of the stable tangent bundle,  then we have  the following twisted index theorem, given by the following commutative diagram
\[\xymatrix{
& K^{ev/odd} (TX, W_3 \circ \tau \circ  \pi ) \ar[dl]_{\ind_t} \ar[dr]^{\ind_a}
&\\
K^t_{ev/odd} (X,  W_3 \circ \tau) \ar[rr]^{\cong}_{\Phi} &  & K^a_{ev/odd} (X, W_3 \circ \tau).}
   \]
This  is a special case of  Connes-Skandalis longitudinal  index theorem for foliations . We will return to this issue later.
 \end{remark}

\section{Geometric  cycles and   geometric   twisted K-homology} \label{section:6}

\begin{definition}
Let $X$ be a paracompact Hausdorff space , and let  $\a:  X \longrightarrow  K(\ZZ, 3)$ be
a twisting over $X$.  A   geometric cycle for $(X, \a)$  is 
a quintuple $(M, \i, \nu, \eta, [E])$ such that
\begin{enumerate}
\item $M$ is a smooth closed manifold  equipped with an $\a$-twisted $Spin^c$ structure;
\[
\xymatrix{M \ar[d]_{\i} \ar[r]^{\nu} & 
\BSO
 \ar@2{-->}[dl]_{\eta} \ar[d]^{W_3} \\
X \ar[r]_\a  & K(\ZZ, 3), } 
\]
where $\i: M\to X$ is a continuous map, $\nu$ is a classifying map of the stable normal bundle, 
and $\eta$ is a homotopy from $W_3\circ \nu$ and $\a \circ \i$;
\item $[E]$ is a K-class in $K^0(M)$ represented by a $\ZZ_2$-graded vector bundle $E$ over $M$. 
\end{enumerate}
Two geometric cycles $(M_1, \i_1, \nu_1, \eta_1, [E_1])$ and $ (M_2, \i,_2 \nu_2, \eta_2, [E_2])$
are isomorphic 
if there is an isomorphism $f:  (M_1, \i_1, \nu_1, \eta_1) \to  (M_2, \i_2,  \nu_2, \eta_2)$,
as $\a$-twisted $Spin^c$ manifolds over $X$,  such that $f_! ([E_1]) = [E_2]$.
\end{definition}

Let $\Gamma (X, \a)$ be the collection of all geometric cycles for $(X,  \a)$. We now impose an equivalence relation $\sim$ on $\Gamma (X, \a)$, generated by the following three elementary 
relations:
\begin{enumerate}
\item  {\bf Direct sum -  disjoint union}

\nin If  $(M , \i , \nu , \eta , [E_1])$ and $ (M , \i,  \nu , \eta , [E_2])$ 
are two geometric cycles with the same $\a$-twisted $Spin^c$ structure,
then 
\[
(M , \i , \nu , \eta , [E_1]) \cup  ( M , \i , \nu , \eta , [E_2]) \sim (M , \i , \nu , \eta , [E_1]+ [E_2]).
\]
\item  {\bf Bordism}

\nin Given  two geometric cycles $(M_1, \i_1, \nu_1, \eta_1, [E_1])$ and $ (M_2, \i_2, \nu_2, \eta_2, [E_2])$, if
there exists a $\a$-twisted $Spin^c$ manifold  $(W, \i, \nu, \eta)$ and $[E]\in K^0(W)$  such that 
\[
\pa (W, \i, \nu, \eta) =  
-(M_1, \i_1, \nu_1, \eta_1) \cup   (M_2, \i_2,  \nu_2, \eta_2)
\]
and $\pa ([E]) = [E_1] \cup [E_2]$. Here $-(M_1, \i_1, \nu_1, \eta_1)$
denotes  the manifold  $M_1$  with the  opposite $\a$-twisted $Spin^c$ structure.

\item   {\bf $Spin^c$ vector bundle modification}

\nin
 Suppose we are given a geometric cycle
 $(M, \i, \nu, \eta, [E]) $ and a  $Spin^c$ vector bundle $V$  over $M$ with 
  even dimensional fibers.  Denote by $\underline{\RR}$ the trivial rank one real 
  vector bundle. Choose a Riemannian metric on $V\oplus \underline{\RR}$, let
  $$\hat{M}= S(V\oplus \underline{\RR})$$  be the sphere bundle of $V\oplus \underline{\RR}$. Then
   the vertical tangent bundle $T^v(\hat{M})$ of $ \hat{M}$ admits a natural $Spin^c$ structure
   with an associated $\ZZ_2$-graded spinor bundle  $S^+_V\oplus S^-_V$ . Denote by
  $\rho: \hat{M} \to M$   the projection  which is   K-oriented.  Then
  \[
  (M, \i, \nu, \eta, [E]) \sim (\hat{M}, \i\circ \rho , \nu \circ \rho, \eta \circ \rho, [\rho^*E\otimes S^+_V]).
  \]
\end{enumerate}

\begin{definition} \label{twisted:geo} Denote by  $K_*^{geo}(X, \a) = \Gamma (X, \a)/\sim$ the 
geometric twisted K-homology. Addition is given by disjoint union - 
direct sum relation. Note that the equivalence relation $\sim$ preserves the parity
of the dimension of the underlying $\a$-twisted $Spin^c$ manifold. Let 
$K^{geo}_{0}(X, \a) $ (resp. $K^{geo}_1(X, \a)$  ) the subgroup of $ K^{geo}_*(X, \a)$
determined by all geometric cycles with even (resp. odd) dimensional
$\a$-twisted $Spin^c$ manifolds. 
\end{definition}

\begin{remark} \begin{enumerate}
\item  According to Proposition \ref{condition},  $M$ admits an $\a$-twisted $Spin^c$ structure
if and only if 
\[
\i^*([\a]) +  W_3(M)=0.
\]
(If $\i$ is an embedding, this is the  anomaly cancellation
condition obtained by Freed and Witten in \cite{FreWit},  
  $ (M, \i, \nu, \eta, [E])$ is  referred to by physicists as a D-brane   appeared  in Type IIB string theory, 
see \cite{FreWit} \cite{Wit1}\cite{Kap}\cite{BouMat}.) 
\item Different definitions of topological twisted K-homology were  proposed in \cite{MatSin}
using $Spin^c$-manifolds and twisted bundles. It  is not
  clear to the author  if their definition is equivalent to Definition \ref{twisted:geo}. 
\item If $f: X\to Y$ is a continuous map and $\a: X\to K(\ZZ, 3)$ is a twisting, then there is a natural
homomorphism  of abelian groups
\[
f_*:  K^{geo}_{ev/odd}(X, \a) \longrightarrow K^{geo}_{ev/odd}(Y, \a\circ f)
\]
sending $[M, \i, \nu, \eta,  E ]$ to $[M,f \circ  \i ,  \nu, \eta,  E]$.   
\end{enumerate}
\end{remark}
 
Given a geometric cycle $(M, \i, \nu, \eta, [E])$, the analytical index (as in Definition \ref{index:analytical})
determines  an element 
\[\begin{array}{lll}
\mu  (M, \i, \nu, \eta, [E]) &  =  & \ind_a(M, \i, \nu, \eta, [E])  \\[2mm]
&=&  \i_*\circ \eta_* \circ I^* \circ PD ([E])\end{array}
\]
in $K^a_{ev/odd} (X, \a)$.

\begin{theorem} \label{Main:2} The assignment $(M, \i, \nu, \eta, [E]) \to \mu  (M, \i, \nu, \eta, [E])  $, called the assembly map,   
defines a natural  homomorphism 
$$\mu: K^{geo}_{ev/odd}(X, \a) \to K^a_{ev/odd} (X, \a)$$
which is an isomorphism for any {\bf smooth}  manifold $X$ with a twisting $\a: X \to K(\ZZ, 3)$.
\end{theorem} 
\begin{proof}{\bf Step 1.}  We need to show that the correspondence is compatible with the three elementary
equivalence relations, so the assembly map  $\mu$ is well-defined. We only need to discuss the even case.

Proposition \ref{a-index:pro} ensures that 
$  \ind_a(M, \i, \nu, \eta, [E]) $ 
is compatible with the bordism  relation and disjoint union - direct sum relation. We only need to check
that the assembly map is compatible with the relation of $Spin^c$ vector bundle
modification.

Suppose given a geometric cycle
$(M, \i, \nu, \eta, [E])$ of even dimension and a  $Spin^c$ vector bundle $V$ over $M$ with 
  even dimensional fibers.   Then
  \[
  (M, \i, \nu, \eta, [E]) \sim (\hat{M}, \i\circ \rho , \nu \circ \rho, \eta \circ \rho, [\rho^*E\otimes S^+_V]).
  \]
 where  $ \hat{M}= S(V\oplus \underline{\RR})$  is  the sphere bundle of $V\oplus \underline{\RR}$
 and   $\rho: \hat{M}\to  M$ is the projection.
  The vertical tangent bundle $T^v(\hat{M})$ of $ \hat{M}$ admits a natural $Spin^c$ structure
   with an associated $\ZZ_2$-graded spinor bundle  $S^+_V\oplus S^-_V$ .   The    K-oriented map $\rho$ induces
  a natural homomorphism (see \cite{AH}) 
  \[
  \rho_!: K^{0}(\hat{M}) \longrightarrow K^0(M)
  \]
  sending $[\rho^*E\otimes S^+_V]$ to $[E]$. This follows from the  Atiyah-Singer index theorem
for  families of longitudinally elliptic differential operator associated to the Dirac operator on the
  round $2n$-dimensional sphere.  Applying  the Poincar\'{e} duality, we have the following
  commutative diagram
  \[
  \xymatrix{
  K^{0}(\hat{M}) \ar[r]^{\rho_!}\ar[d]^{PD} &  K^0(M)\ar[d]^{PD}  \\
  K^a_0(\hat{M}, W_3\circ \tau \circ \rho) \ar[r]^{\rho_*}& K^a_0(M, W_3\circ \tau),}
  \]
  which implies that $PD([E]) = \rho_* \circ PD([\rho^*E\otimes S^+_V])$. Hence,  we have 
  \[
  \mu  (M, \i, \nu, \eta, [E]) = \mu  (\hat{M}, \i\circ \rho , \nu \circ \rho, \eta \circ \rho, [\rho^*E\otimes S^+_V]).
  \]
  
  {\bf Step 2.} We  establish   the following commutative diagram 
  \[
  \xymatrix{ 
& K^t_{ev/odd} (X , \a )  \ar[dl]_{\Psi} \ar[dr]^{\Phi}_{\cong}
&\\
K^{geo}_{ev/odd}(X, \a ) \ar[rr]^{\mu} & & K^a_{ev/odd} (X, \a)}
\]
and show that $\Psi$ is surjective. This  implies 
 that $\mu$ is an isomorphism. 
 
 Firstly, we construct a natural map  $\Psi: K^t_{ev} (X , \a ) \to K^{geo}_0(X, \a)$. Given
 an element of  $K^t_0 (X , \a )$ represented by a map
 \[
 \theta: (S^{m+2k}, \infty) \to (\cP_\a (\KK)/X, *)
 \]
 for a sufficiently large $k$.   We can lift this map to  a map 
 $\theta_0:  S^{m+2k} \to \cP_\a ( \MSpin (2k)) /X$. As in Step 4 of  the proof of Theorem \ref{P-Thom},   $\theta_0$ can be deformed  to a differentiable map $h$ on the preimage of 
 some open set containing  $\cP_\a ( \BSpin^c (2k)) $,  is transverse to  $\cP_\a ( \BSpin^c (2k))$
 and agrees  with $\theta_0 $ on an open set containing $\infty$. Then 
 \[
 M= h^{-1}  (\cP_\a ( \BSpin^c (2k))) \subset \RR^{m+2k} =  S^{m+2k} -\{\infty\}
 \]
 admits a natural $\a$-twisted $Spin^c$   structure 
\[
 \xymatrix{
    M   \ar@/{}_{1pc}/[ddr]_\i  \ar@/{}^{1pc}/[drr]^\nu
       \ar[dr]^{h|_M }            \\
      &   \cP_\a (\BSpin^c (2k) )   \ar[d]  \ar[r]    
                     & \BSO \ar@2{-->}[dl]  \ar[d]^{W_3}       \\
      & X \ar[r]_\a   & K(\ZZ, 3).            }
    \]
A homotopy equivalence map gives rise to a bordant $\a$-twisted $Spin^c$  manifold. 
Hence, we have a geometric cycle $(M, \i, \nu, \eta, [\underline{\CC}])$, whose equivalence class doesn't
depend on various choices in the construction. This defines  a map
\[
\Psi:  K^t_0 (X , \a ) \longrightarrow K^{geo}_0(X, \a ).
\]
It is straightforward to show that $\Psi$ is a homomorphism.  Note that $\Phi = \mu \circ \Psi$ follows from the definition of   $\Phi$ and Theorem \ref{Main}. 

To show that $\Psi$ is surjective, let $(M, \i, \nu, \eta, [E])$ be a geometric cycle. Then the 
$\a$-twisted $Spin^c$  manifold $(M, \i, \nu, \eta)$ defines a  bordism class
in $\Omega^{Spin^c}_{ev}(X, \a)$.  The topological index
\[
\ind_t (M, \i, \nu, \eta) \in K^t_0(X, \a)
\]
is represented by the canonical map
\[
\theta: (S^{m+2k}, \infty)  \to ( \Th (N_M), *) \to( \cP_{\a} (MSpin (2k) /X, *) \to (\cP_\a ( \KK ) /X, *)
\]
 associated to the normal bundle $\pi: N_M\to M$ of an embedding 
$i_k: M \to \RR^{m+2k}$ as in  Step 1 of the proof of Theorem \ref{Main}. This map defines a compactly supported 
section of $\cP_{\a\circ \i \circ \pi} (\KK)$ (a bundle of K-theory spectra over $N_M$). We also denote this section by
$\theta$.  Then the homotopy class of the section $\theta$ defines a twisted K-class
in $K^0(N_M, \a\circ \i\circ \pi)$, which is mapped to $[\underline{\CC}]$ under the Thom isomorphism
\[\begin{array}{lll}
K^0(N_M, \a\circ \i\circ \pi) &\cong &  K^0(N_M, W_3\circ \nu\circ \pi)  \\[2mm]
&\cong &   K^0(M).
\end{array}
\]

Let $\sigma: M\to \KK$ be a map representing the K-class $[E]$.  
$\sigma \circ \pi$ is a section of the trivial bundle  $\underline{\KK}$ over $N_M$. Define a new section
of $\cP_{\a\circ \i \circ \pi} (\KK)$ by applying the fiberwise multiplication $m: \KK \wedge \KK\to \KK$
to  $(\theta, \sigma\circ \pi)$. Then $m(\theta, \sigma\circ \pi)$ is a compactly supported section 
of $\cP_{\a\circ \i \circ \pi} (\KK)$ which determines a map, denoted by $\theta\cdot \sigma$
\[
\theta\cdot \sigma: (S^{m+2k}, \infty)  \to ( \Th (N_M), *)  \to (\cP_{\a\circ \i \circ \pi} (\KK)/N_M, *).
\]
The   homotopy class of $\theta\cdot \sigma$ as an element in $K^0(N_M , \a \circ \i\circ \pi ) $,  is uniquely determined by the stable homotopy  class of $\theta$ and  the homotopy class of $\sigma$.   Under the Thom isomorphism
$K^0(N_M, \a\circ \i\circ \pi) \cong  K^0(M)$, $[\theta\cdot \sigma]$ is mapped to $[E]$.  Hence, 
\[
\Psi([\theta\cdot \sigma]) = [M, \i, \nu, \eta, [E]].
\]
Therefore, $\Psi$ is surjective.

 \end{proof}
 
\begin{corollary} \label{D:brane}
 Given a twisting $\a: X\to K(\ZZ, 3)$ on a smooth manifold $X$,  every twisted
K-class in $K^{ev/odd}(X, \a)$ is represented by a geometric cycle
supported on an $(\a + ( W_3 \circ \tau))$-twisted closed $Spin^c$-manifold
$M$ and an ordinary  K-class $[E] \in K^0(M)$.
\end{corollary}

 \begin{proof} We only need to prove the even  case, the odd case can be obtained by the suspension operation.  Assume that  $X$ is even dimensional and $\pi: TX \to X$ is the projection,  then we have the following isomorphisms
 {\small\[
 \begin{array}{lllll}
 K^0(X, \a) &\cong&  K^0 (TX, (\a\circ \pi) + ( W_3 \circ \tau\circ \pi)  ) & \quad & \text{(Thom isomorphism)} \\[2mm]
 & \cong &  K^a_0  (X, \a    + ( W_3 \circ \tau ) ) & \quad & \text{(Remark \ref{remark:index})} \\[2mm]
 & \cong &  K^{geo}_0 (X, \a   + ( W_3 \circ \tau ))  & \quad  &\text{(Theorem \ref{Main:2})}\\[2mm]
 \end{array} 
 \]}
 From the definition of $K^{geo}_0 (X, \a   + ( W_3 \circ \tau )) $, we know that each element in
 $K^{geo}_0 (X, \a   + ( W_3 \circ \tau ))$ is represented by a geometric cycle 
 $(M, \i, \nu, \eta, [E])$  for
 $(X, \a    + (W_3 \circ \tau ) )$, which is a  generalized D-brane
supported on an $(\a  + ( W_3 \circ \tau)) $-twisted closed $Spin^c$-manifold
$M$ and an ordinary  K-class $[E] \in K^0(M)$.
  \end{proof}

\begin{remark}Let $Y$ be a closed subspace of $X$.   A relative   geometric cycle for $(X, Y;  \a)$  is  a 
quintuple   $(M, \i, \nu, \eta, [E])$    such that
\begin{enumerate}
\item $M$ is a smooth  manifold  (possibly with boundary), equipped  with an $\a$-twisted $Spin^c$ structure $(M, \i, \nu, \eta)$;
\item if $M$ has a non-empty boundary, then  $\i (\pa M) \subset  Y$;
\item $[E]$ is a K-class in $K^0(M)$ represented by a $\ZZ_2$-graded vector bundle $E$ over $M$, or
a continuous map  $M \to \KK$. 
\end{enumerate}
The  relation $\sim$ generated by disjoint union - direct sum,  bordism
and $Spin^c$ vector bundle
modification is an equivalence relation.  The collection of relative geometric cycles, modulo
the equivalence relation   is denoted by
\[K^{geo}_{ev/odd}(X, Y; \a ).
\]
Then we have the following  commutative diagram whose arrows are all isomorphisms
\[
  \xymatrix{ 
& K^t_{ev/odd} (X ,Y;  \a )  \ar[dl]_{\Psi} \ar[dr]^{\Phi} 
&\\
K^{geo}_{ev/odd}(X, Y; \a ) \ar[rr]^{\mu} & & K^a_{ev/odd} (X,Y; \a).}
\]

\end{remark}

 \section{The twisted  longitudinal index theorem for foliation}   \label{section:app}

 Given a $C^\infty$  foliated manifold $(X, F)$, that is, $F$ is an integrable sub-bundle of $TX$, let $D$
 be an elliptic differentiable operator along the leaves of the foliation. Denote by $\sigma_D$
 the longitudinal symbol of $D$, whose class in $K^0(F^*)$ is denoted by $[\sigma_D]$. 
 In \cite{ConSka}, Connes and Skandalis defined  the topological index and the analytical
 index of $D$  taking values in the K-theory of the  foliation $C^*$-algebra $C_r^*(X, F)$ and established
 the equality between the topological index and the analytical
 index of $D$. See \cite{ConSka} for more details.   In this section, we will
 generalize the Connes-Skandalis  longitudinal  index theorem  to a 
  foliated manifold $(X, F)$ with a twisting $\a: X\to K(\ZZ, 3)$. 
 
   Let   $N_F = TX/F$
  be the normal bundle to the leaves whose
  classifying map is denoted by $\nu_F: X\to BSO (k)$. Here assume 
  that $F$ is rank $k$ and $X$ is $n$ dimensional.  We can equip $X$ with a Riemannian metric
  such that we have a splitting 
  \[
  TX = F \oplus N_F.
  \]
  Then the sphere bundle $M=S(F^*\oplus \underline{\RR})$ is a
  $W_3 \circ \nu_F$-twisted $Spin^c$ manifold over $X$. Let $\pi$ be  the projection
  $M\to X$.  To see this, we need to calculate the third Stiefel-Whitney class of $M=S(F^*\oplus \underline{\RR})$ from the following exact sequence  of bundles over $M$
  \ba\label{split}
  0\to \pi^* (F\oplus \underline{\RR}) \longrightarrow TM \oplus \underline{\RR}  \longrightarrow
  \pi^*TX \to 0,
  \na
  from which we have 
  \[
  \begin{array}{lll}
  W_3(TM) &=& \pi^*W_3(F) + \pi^* W_3(TX)\\[2mm]
  &=& \pi^* W_3(F) + \pi^* (W_3(F) +W_3(N_F))\\[2mm]
  &=& \pi^*W_3(N_F).
  \end{array}
  \]
  Note that $\pi^* W_3(N_F) = [W_3\circ \nu_F \circ \pi]\in H^3(M, \ZZ)$. 
  So $S(F^*\oplus \underline{\RR})$
  admits a natural $W_3 \circ \nu_F$-twisted $Spin^c$ structure
\[
\xymatrix{M \ar[d]_{\pi } \ar[r]^{\nu} &
\BSO
 \ar@2{-->}[dl]_{\eta} \ar[d]^{W_3} \\
X \ar[r]_{W_3 \circ \nu_F}  & K(\ZZ, 3), } 
\]
where $\nu$ is the classifying map of the stable normal of $M$ and $\eta$ is a homotopy associated
to a splitting of (\ref{split}) as follows. Given a splitting of (\ref{split}), the natural isomorphisms
\[\begin{array}{lll}
TM \oplus \underline{\RR}  &  \cong &  \pi^* TX \oplus  \pi^*( F \oplus \underline{\RR}) \\[2mm]
&\cong & \pi^* ( F \oplus N_F) \oplus  \pi^*( F \oplus \underline{\RR}) \\[2mm]
&\cong & \pi^* ( F \oplus  F) \oplus \pi^* N_F \oplus \underline{\RR}
\end{array}
\]
and the canonical $Spin^c$ structure on $ \pi^* ( F \oplus  F)$ define the homotopy 
between $W_3\circ \tau$ and $W_3 \circ \nu_F$. Different choices  of   splittings   of (\ref{split})
gives rise to the same  homotopy equivalence class, hence doesn't change the twisted $Spin^c$ bordism class of $M$.

Given an elliptic differentiable operator along the leaves of the foliation with 
longitudinal symbol class $[\sigma_D]\in K^0(F^*)$  represented by
a map
\[
\sigma_D: \pi^*E_1 \to \pi^*E_2
\]
of a pair of vector bundles  $E_1$ and $E_2$ over $X$ such that $\sigma_D$ is an isomorphism
away from the zero section of $F^*$. Applying the clutching  construction as described in \cite{BD1},
$M=  S(F^*\oplus \underline{\RR})$ consists of two copies of the unit ball bundle
of $F^*$ glued together by the identity map of $S(F^*)$. We form a vector bundle
over $M$ by gluing $\pi^*E_1$ and $\pi^*E_2$  over  each copy of the unit ball bundle 
along $S(F^*)$ by the symbol map $\sigma_D$. Denote the resulting vector bundle
by $\hat{E}$.  The quintuple $(M, \pi, \nu, \eta, [\hat{E}])$ is a geometric cycle of $(X, W_3 \circ \nu_F)$.

We define the topological index of $[\sigma(D)]$ to be
\[
\ind_t ([\sigma_D]) = [M, \pi, \nu, \eta, \hat{E}] \in K^{geo}_{[n+k]}(X, W_3 \circ \nu_F)
\]
where $[n+k]$ denotes the mod 2 sum    (even or odd  if $n+k$ is even  or  odd). 

The analytical index of $[\sigma_D]$  is defined through 
the following sequence of isomorphisms   
 \[\begin{array}{lllll}
K^0(F^*) &\cong & K^{[k]} (X, W_3(F) )&\quad &  (\text{Thom isomorphism})\\[2mm]
&\cong & K^a_{[n+k]} (X, W_3(F\oplus TX) )&\quad &  (\text{Poincar\'{e} duality})\\[2mm]
&\cong &  K^a_{[n+k]} (X, W_3(N_F) )&\quad &  (F\oplus TX \cong F\oplus F \oplus N_F) \\[2mm]
&\cong & K^a_{[n+k]} (X, W_3\circ \nu_F).   &&\end{array}
\]
The resulting element is denoted by $$\ind_a ([\sigma_D]) \in K^a_{[n+k]} (X, W_3\circ \nu_F).$$

  Now we apply Theorems   \ref{Main} and    \ref{Main:2},   and Remark \ref{remark:index} to obtain the following version of the 
 longitudinal index theorem  for the foliated manifold $(X, F)$.  This  longitudinal  index theorem  is equivalent to 
 the Connes-Skandalis   longitudinal index theorem  through the  natural homomorphism 
 $$ K^a_{[n-k]} (X, W_3 \circ \nu_F)\to   K_0(C_r^*(X, F)).$$ 
 
 \begin{theorem} Given a $C^\infty$   n-dimensional foliated manifold $(X, F)$ 
 of rank $k$,    the longitudinal index theorem for  $(X, F)$ is given by the following
 commutative diagram 
\[
   \xymatrix{ 
& K^0 (F^*)  \ar[dl]_{\ind_t} \ar[dr]^{\ind_a} 
&\\
K^{geo}_{[n-k]}(X, W_3 \circ \nu_F) \ar[rr]^{\mu}_{\cong} & & K^a_{[n-k]} (X, W_3 \circ \nu_F ),}
\]
 whose arrows are all   isomorphisms. 
\end{theorem}
 
 \begin{remark}\label{index:families} 
 If $(X, F)$ comes from a fibration $\pi_B: X \to B$ such that the leaves are the fibers of $\pi_B$, then
 $F $ is given by  the vertical tangent bundle  $T(X/B)$ and $N_F \cong \pi_B^* TB$. This isomorphism
 defines  a canonical
 homotopy $\eta_0$  realizing $W_3 \circ \nu_F \sim W_3\circ \tau_B\circ \pi_B$, 
 where  $\tau_B$ is the classifying map of the stable tangent bundle of $B$.  
  The following homotopy diagram
 \[
  \xymatrix{S(F^*\oplus \underline{\RR})  \ar[d]_{\pi } \ar[rr]^{\nu} & &
\BSO
 \ar@2{-->}[dll]_{\eta} \ar[d]^{W_3} \\
  X\ar[d]_{\pi} \ar[rr]^{W_3 \circ \nu_F} &\ar@2{-->}[dl]_{\eta_0} & K(\ZZ, 3),    \\
  B \ar[rru]_{W_3\circ \tau_B} && } 
\]
   implies
 that   $(S(F^*\oplus \underline{\RR}), \pi_B\circ \pi, \nu, \eta *\eta_0, [\hat{E}])$,
 where $\eta *\eta_0$ is the obvious homotopy joining $\eta$ and $\eta_0$,
   is a geometric cycle of 
 $(B, W_3 \circ \tau_B)$ and 
  \[
( \pi_B)_! (S(F^*\oplus \underline{\RR}), \pi, \nu, \eta, [\hat{E}]) = (S(F^*\oplus \underline{\RR}), \pi_B\circ \pi, \nu, \eta *\eta_0,[\hat{E}]).
\]
 The  commutative diagram
 \[
   \xymatrix{ 
& K^0 (F^*)  \ar[dl]_{\ind_t} \ar[dr]^{\ind_a} 
&\\
K^{geo}_{[n-k]}(X, W_3 \circ \nu_F) \ar[d]^{( \pi_B)_! } \ar[rr]^{\mu} & & K^a_{[n-k]} (X, W_3 \circ \nu_F )
\ar[d]^{( \pi_B)_! } \\
K^{geo}_{[n-k]}(B, W_3\circ \tau_B ) \ar[rr]^{\mu} \ar[dr]_{PD}  & & K^a_{[n-k]} (B, W_3 \circ \tau_B )\ar[dl]^{PD} \\
& K^0 (B )  
&}
\]
becomes  the  Atiyah-Singer   families index theorem in \cite{AtiSin4}. 
 \end{remark}
 
 In the presence of a twisting $\a: X\to K(\ZZ, 3)$ on a foliated manifold $(X, F)$,    
 Theorems   \ref{Main} and  \ref{Main:2},   and Remark \ref{remark:index}   give  rise to  
 the following twisted longitudinal index theorem.
 
 \begin{theorem} \label{index:foliated} Given a $C^\infty$   n-dimensional foliated manifold $(X, F)$ 
 of rank $k$ and a twisting $\a: X\to K(\ZZ, 3)$,  let $\pi: F^* \to X$ be the projection. Then
    the twisted  longitudinal index theorem for the  foliated manifold 
 $(X, F)$ with a twisting $\a$  is given by the following
 commutative diagram 
\[
   \xymatrix{ 
& K^0 (F^*, \a\circ \pi )  \ar[dl]_{\ind_t} \ar[dr]^{\ind_a} 
&\\
K^{geo}_{[n-k]}(X, \a + (W_3 \circ \nu_F))  \ar[rr]^{\mu}_{\cong} & & K^a_{[n-k]} (X, \a + (W_3 \circ \nu_F) ),}
\]
 whose arrows are all   isomorphisms.  In particular,  if  $(X, F)$ comes from a fibration $\pi_B: X \to B$ and  a   twisting $\a\circ \pi_B$ on $X$ comes from a twisting $\a$ on $B$ , then we have the following twisted version of the 
 Atiyah-Singer families index theorem with  notations from Remark \ref{index:families} 
\small{ \[
 \xymatrix{ 
& K^0 (T^*(X/B),  \a\circ \pi_B \circ \pi   )  \ar[dl]_{( \pi_B)_! \circ \ind_t \ } \ar[dr]^{( \pi_B)_! \circ\ind_a} 
&\\
K^{geo}_{[n-k]}(B, \a + (W_3\circ \tau_B )) \ar[rr]^{\mu} \ar[dr]_{PD}  & & K^a_{[n-k]} (B, \a+ (W_3 \circ \tau_B ))\ar[dl]^{PD} \\
& K^0 (B, \a ) . 
&}
\]}
 \end{theorem}

 In \cite{MMSin}, Mathai-Melrose-Singer established  the index theorem for projective families of 
 longitudinally  elliptic operators associated to 
 a fibration $\phi: Z \to X$ and an Azumaya bundle $\cA_\a$
 for $\a$ representing a torsion class  in $ H^3(X, \ZZ)$. 
 
 Given  a  local  trivialization of $\cA_a$  for an
  open covering of $X=\cup_{i}U_i$,  according to \cite{MMSin},  a projective family of 
  longitudinally    elliptic operators is a collection of 
   longitudinally    elliptic pseudo-differential operators,  
   acting on finite dimensional   vector bundles of fixed rank over each of the  open sets 
   $\{\phi^{-1}(U_i) \}$  such
   that the compatibility condition over triple overlaps may fail by a scalar factor. The symbol class
   of such a  projective family of   elliptic operators  determines  a class in
   \[
   K^0(T^*(Z/X), \a\circ \phi \circ  \pi),
   \]
   where  $T^*(Z/X)$ is  dual to  the vertical tangent bundle  of $Z$ and
    $\pi: T^*(Z/X) \to Z$ is the projection.     Let $n$ be the dimension of $Z$ and $k$ be the
    dimension of the fiber of $\phi$. 
  The Thom isomorphism, Theorems \ref{Main}   and \ref{Main:2}  give  rise to the following
  commutative diagram
  \[
 \xymatrix{ 
& K^0 ( T^*(Z/X), \a\circ \phi \circ  \pi)  \ar[dl]_{\phi_!\circ \ind_t} \ar[dr]^{\phi_!\circ \ind_a} &\\ 
K^{geo}_{[n-k]}(X , \a  + ( W_3\circ \tau ) ) \ar[rr]^{\mu} \ar[dr]_{PD}  & & K^a_{[n-k]} (X , \a  + ( W_3\circ \tau ))
\ar[dl]^{PD}_{\cong} \\
& K^0 (X, \a )
& }
\]
here $\a  + ( W_3\circ \tau) $ represents the class $[\a]+W_3(X) \in H^3(X, \ZZ)$. 
Readers familiar with \cite{MMSin} will recognise that the above theorem 
is another way of writing the Mathai-Melrose-Singer  index theorem 
(Cf. Theorem 4  in \cite{MMSin}) for projective families of 
 longitudinally  elliptic operators associated a fibration 
 $\phi: Z \to X$ and an Azumaya bundle $\cA_\a$
 for $\a$ representing a torsion class  $[\a]\in H^3(X, \ZZ)$.

  \section{Final remarks} \label{section:8}

  Let $M$ be  an oriented  manifold   with a map $\nu: M \to \BSO $ classifying its  stable
   normal bundle. Given any fibration $\pi: B \to BSO$, we can define a $B$-structure on
   $M$ to be a homotopy class of lifts $\tilde \nu$  of $\nu$:
   \ba\label{lift}
   \xymatrix{
 &    B \ar[d]^\pi  \\
   M \ar@{-->}[ur]^{\tilde\nu}\ar[r]^\nu &  \BSO. 
   }
   \na
   When $B$ is $\BSpin^c$, then a lift   $\tilde{\nu }$  in (\ref{lift}) is a $Spin^c$ structure on its
   stable normal bundle.  
    When $B$ is $\BSpin$,  then a lift   $\tilde{\nu }$  in (\ref{lift}) is a $Spin$ structure on its
   stable normal bundle.

  Define
   $$\String = \disp{\vlim_{k\to \infty}} \String (k)$$ 
  where $\String (k)$ is an infinite dimensional
   topological group constructed in \cite{StoTei}. There is a map $\String(k) \to Spin(k)$ which
   induces an isomorphism $\pi_n (\String(k) ) \cong \pi_n(Spin(k))$ for all $n$ except $n=3$ when
   $\pi_3(\String(k) ) =0 $ and $\pi_3(Spin(k) ) \cong \ZZ$.

   Let $M$ be  a $Spin$  manifold   with a classifying map $\nu: M \to \BSpin $ for the $Spin$
   structure on  its  stable normal bundle.  A  string structure on $M$ is a lift $\tilde \nu$  of $\nu$:
   \[
   \xymatrix{
 &    \BString  \ar[d]^\pi  \\
   M \ar@{-->}[ur]^{\tilde\nu}\ar[r]^\nu &  \BSpin. 
   }
   \]
   We point out that an oriented manifold $M$ admits a spin structure on its stable normal bundle
   if and only if its second Stiefel-Whitney class $w_2(M)$  vanishes, and a spin manifold $M$  admits a string
   structure on its stable  normal bundle if  $\frac{p_1(M)}{2}$ vanishes, where  $p_1(M)$ denotes the first 
   Pontrjagin class of  $M$ (Cf. \cite{McL} \cite{StoTei}).   If $M$ is a string manifold, then $M$ has
   a canonical orientation with respect to elliptic cohomology.

 The tower of Eilenberg-MacLane  fibrations 
   \[\xymatrix{
   \BString\ar[d]^{K(\ZZ, 3)} & & \\
   \BSpin \ar[d]^{K(\ZZ_2, 1)} \ar[rr]^{\frac{p_1}{2}} && K(\ZZ, 4)\\
   \BSO  \ar[rr]^{w_2} && K(\ZZ_2, 2)}
   \]
  gives  rise to Thom 
  spectra
  \[
  \MString \longrightarrow \MSp   \longrightarrow \MSO,
  \]
 with  corresponding bordism groups 
  \[
  \Omega_*^{String}(X)   \longrightarrow \Omega_*^{Spin}(X)  \longrightarrow\Omega_*^{SO}(X). 
  \]
  
 \begin{remark}
   Given a paracompact space $X$, a continuous map $\a: X \to K(\ZZ_2, 2)$ is called
   a KO-twisting, and a continuous map $\a: X \to K(\ZZ, 4)$ is called a string twisting.  
   For a principal $G$-bundle $\cP$ over $X$ for a compact Lie group $G$ equipped
   with a map $BG\to K(\ZZ, 4)$ representing a  degree 4 class in $H^4(BG, \ZZ)$,  there
   is a natural string twisting
    \[
   X\longrightarrow BG \longrightarrow K(\ZZ, 4).
   \]
   Given a string twisting $\a: X \to K(\ZZ, 4)$, a universal Chern-Simons 2-gerbe
   was constructed in \cite{CJMSW}.
\end{remark}

   For   any KO-twisting $\a$,  there is a corresponding notion
  of an $\a$-twisted $Spin$ manifold  over $(X, \a)$.
  
  \begin{definition} \label{twisted:spin}  Let $(X, \a)$ be  a paracompact  topological  space   with a
   twisting 
$\a: X\to K(\ZZ_2,2)$.   An $\a$-twisted $Spin$ manifold   over $X$ is   a  quadruple $(M, \nu,  \i, \eta)$ 
where 
\begin{enumerate}
\item $M$ is a smooth,  oriented and {\bf compact}  manifold   together  with a fixed classifying map of its  stable normal bundle  
\[
\nu:   M \longrightarrow  \BSO.
\]
\item  $\i: M\to X$ is a  continuous map;
\item $\eta$ is an $\a$-twisted $Spin$  structure  on $M$, that is   a homotopy commutative diagram
\[
\xymatrix{M \ar[d]_{\i} \ar[r]^{\nu} &
\BSO
 \ar@2{-->}[dl]_{\eta} \ar[d]^{w_2} \\
X \ar[r]_\a  & K(\ZZ_2, 2), } 
\]
where $w_2$ is the classifying map of the principal 
$K(\ZZ_2, 1)$-bundle $\BSpin  \to \BSO$  associated to
the second   Stiefel-Whitney class and   $\eta$ is a homotopy
between $w_2 \circ \nu $ and $\a \circ \i$.  
\end{enumerate}
Two  $\a$-twisted $Spin $  structures  $\eta$ and
$\eta'$ on $M$ are called equivalent if there is a homotopy between $\eta$ and $\eta'$.
\end{definition}

 \begin{remark}  Given a smooth,  oriented and compact  n-dimensional manifold $M$
  and a paracompact space $X$ with a  KO-twisting
$\a: X\to K(\ZZ_2, 2).$
\begin{enumerate}  \item 
$M$ admits an $\a$-twisted $Spin$ structure if and only if  there exists
a continuous map
$\i: M\to X$ such that 
\ba\label{cond:spin}
\i^*([\a]) + w_2(M)=0
\na
in $H^2(M, \ZZ_2)$, here $w_2(M)$ is the second  Stiefel-Whitney class of $TM$. (The condition
(\ref{cond:spin}) is 
the anomaly  cancellation condition for Type I D-branes (Cf. \cite{Wit3}).)
\item If $\i^*([\a]) + w_2(M) =0$, then the  set  of equivalence classes of $\a$-twisted $Spin $ structures
on $M$ are in one-to-one correspondence with elements in $H^1(M, \ZZ_2)$.
\end{enumerate}
\end{remark}

 Let $\cH_\RR$ be an  infinite dimensional, real and separable Hilbert space.  The projective orthogonal  group $PO(\cH_\RR)$ with the norm topology (Cf. \cite{Kui})  has the homotopy type of an 
Eilenberg-MacLane space $K(\ZZ_2, 1)$. 
The  classifying space of $PO(\cH_\RR)$, as  a classifying space of principal $PO(\cH_\RR)$-bundle,  is a $K(\ZZ_2, 2)$. 
Thus, the set of isomorphism classes of  locally trivial 
principal $PO(\cH_\RR)$-bundles over   $X$ is canonically identified
with   
$$[X, K(\ZZ_2, 2)] \cong H^2(X, \ZZ_2).$$   Given  a KO-twisting $\a: X\to K(\ZZ_2, 2)$, there is a canonical  
principal $K(\ZZ_2, 1)$-bundle  over   $X$, or equivalently, a  locally trivial 
principal $PO(\cH_\RR)$-bundle  $\cP_\a$   over   $X$.   Let   $\cK_\RR$ be the $C^*$-algebra of real compact operators on
$\cH_\RR$. Denote by  $C_c(X, \cP_\a (\cK_\RR))$ the $C^*$-algebra of compactly supported sections
of the associated bundle 
\[
 \cP_\a (\cK_\RR) :=  \cP_\a \times _{PO(\cH_\RR)} \cK_\RR.
 \]
 In \cite{Ros} (see also \cite{MMSte}), twisted KO-theory is defined for   $X$ with a KO-twisting
  $$\a: X \to K(\ZZ_2, 2)$$   to be
  \[
  KO^i(X, \a) := KO_i\bigl(C_c(X, \cP_\a (\cK_\RR))\bigr),
  \]
 Let $\KK_\RR$ be the 0-th space of  the KO-theory spectrum, then 
 there is a base-point preserving action
of $K(\ZZ_2, 1)$ on the real  K-theory spectrum
\[
K(\ZZ_2, 1) \times \KK O \longrightarrow  \KK O
\]
which is represented by the action of real line bundles on ordinary KO-groups.  
This action defines an associated bundle of  KO-theory spectrum over $X$. Denote 
\[
\cP_\a (\KK_\RR) = \cP_\a\times_{K(\ZZ_2, 1)} \KK_\RR 
\]
the bundle of based spectra over $X$ with fiber the  KO-theory spectra,  and $\{ \Omega^n_X \cP_\a(\KK_\RR) = 
 \cP_\a\times_{K(\ZZ_2, 1)} \Omega^n \KK_\RR  \}$  the   fiber-wise iterated loop spaces. 
 Then we have an equivalent definition of twisted KO-groups  of $(X, \a)$ (Cf. \cite{Ros})
\[
KO^n(X, \a) = \pi_0\bigl( C_c(X, \Omega^n_X \cP_\a(\KK_\RR))\bigr)
\]
the  set of homotopy classes of compactly supported sections of the bundle of K-spectra.
Due to Bott periodicity, we only have eight different  twisted K-groups $KO^i(X, \a)$ ($ i=0, \cdots 7$).
Twisted KO-theory is a 8-periodic generalized cohomology theory.

One would expect that results in this paper can be extended  to twisted KO-theory. Much of the
constructions  and arguments in this paper go through in the case of   twisted KO-theory. The subtlety is
to study the twisted $KR$-theory for the (co)-tangent bundle with the canonical involution.  
  This may requires additional arguments. 
  
  Another interesting generalization is the notion of twisted string structure  for a  paracompact  topological  space $X$ with
  a string twisting given by
  \[
  \a: X  \longrightarrow  K(\ZZ, 4).
  \]
    
\begin{definition} \label{twisted:str}  Let $(X, \a)$ be  a paracompact  topological  space   with a
string  twisting 
$\a: X\to K(\ZZ, 4)$.   An $\a$-twisted {\bf string }  manifold   over $X$ is  a  quadruple $(M, \nu,  \i, \eta)$ 
where 
\begin{enumerate}
\item $M$ is a smooth  compact  manifold   with a stable spin structure on its    normal bundle  given
by 
\[
\nu:   M \longrightarrow  \BSpin
\]
here $\BSpin  = \disp{ \varinjlim_{k } } \BSpin (k)$ is  the classifying space of the stable  spin structure;
\item  $\i: M\to X$ is a continuous map;
\item $\eta$ is an $\a$-twisted string structure   on $M$, that is   a homotopy commutative diagram
\[
\xymatrix{M \ar[d]_{\i} \ar[r]^{\nu} &
\BSpin
 \ar@2{-->}[dl]_{\eta} \ar[d]^{\frac{p_1}{2} } \\
X \ar[r]_\a  & K(\ZZ, 4), } 
\]
where $\frac{p_1}{2}: \BSpin \to K(\ZZ, 4) $ is the classifying map of the  principal 
$K(\ZZ, 3)$-bundle $\BString  \to \BSpin $,  representing the generator of $H^4(\BSpin, \ZZ)$,
  and  $\eta$ is a homotopy
between $\dfrac{p_1}{2}  \circ \nu $ and $\a \circ \i$.  
\end{enumerate}
 Two  $\a$-twisted String   structures  $\eta$ and
$\eta'$ on $M$ are called equivalent if there is a homotopy between $\eta$ and $\eta'$.
\end{definition}  

\begin{remark} 
Given a smooth compact  spin manifold $M$ and a paracompact space $X$ with a string twisting
$\a: X\to K(\ZZ, 4).  $
\begin{enumerate}  \item 
$M$ admits an $\a$-twisted string  structure if and only if  there is a continuous map
$\i: M\to X$ such that
\ba\label{cond:string}
\i^*([\a])  +  \frac{p_1 (M)}{2}  =0
\na
in $H^4(M, \ZZ)$, here $ p_1(X) $ is the first Pontrjagin class of $TM$. 
 \item If $\i^*([\a])  +  \frac{p_1 (M)}{2} =0$,  then the  set  of equivalence classes of $\a$-twisted 
string  structures
on $M$ are in one-to-one correspondence with elements in $H^3(M, \ZZ)$.
\end{enumerate}
\end{remark}

Given a manifold $X$ with a twisting $\a: X\to K(\ZZ, 4)$,  one can form a    bordism category,
 called the $\a$-twisted string bordism over $(X, \a)$,  whose 
 objects are compact smooth  spin  manifolds   over $X$ with an $\a$-twisted string  structure. 
  The corresponding bordism  group $\Omega_*^{String}(X, \a)$ is called the  $\a$-twisted string
    bordism  group of $X$.   We will study these  $\a$-twisted string
    bordism  groups and their applications    elsewhere.  

\vskip .2in
\noindent
{\bf Acknowledgments}  

The author  likes to thank Paul Baum,  Alan Carey, Matilde Marcolli, Jouko Mickelsson,  Michael Murray,
Thomas Schick and Adam Rennie   for  useful conversations and their comments on the manuscript.  The author thanks  Alan Carey for his continuous support and encouragement.  The author also thanks the referee for the suggestions to improve the manuscript.  The  work is  supported in part by Carey-Marcolli-Murray's ARC  Discovery Project  DP0769986.

  \end{document}